\documentclass[11pt]{amsart}
\usepackage[margin=3.3 cm]{geometry}
\usepackage{latexsym}
\usepackage{amsfonts}
\usepackage{amssymb,amscd}
\usepackage{amsmath}
\usepackage{graphics, color}
\usepackage{hyperref}

\def\r{\mathcal R}
\def\R{\mathbb R}
\def\s{\mathbb S}
\def\C{\mathbb C}

\def\F{\mathrm F}
\def\H{\mathbb H}
\def\L{{\rm L}}

\def\z{{\bf z}}
\def\w{{\bf w}}

\def\PSp{ {\mathrm{ PSp}}}

\def\V{\mathrm V}

\def\U{{\rm U}}
\def\X{\mathbb X}
\def\A{\mathbb A}

\def \ch{{\bf H}_{\C}}
\def \h{{\bf H}_{\H}}
\def \hh{{\bf H}_{\H}}
\def\R{\mathbb R}

\def\p{\mathcal P}

\newcommand{\SL}{\mathrm{SL}}
\newcommand{\SU}{\mathrm{SU}}
\newcommand{\GL}{\mathrm{GL}}
\newcommand{\Sp}{\mathrm{Sp}}

\def\P{\mathbb P}
\def\M{{\mathrm M}}

\def \tr {{\rm tr}_{\R}}

\def \a {{\bf a}}
\def \r  {{\bf r}}
\def \x {{\bf x}}

\def \p {{\bf p}}
\def \q  {{\bf q}}
\def \i {{\bf i}}
\def \j {{\bf j}}
\def \k {{\bf k}}
\def \Tr {{\rm Tr}}

\newcommand{\fX}{\mathfrak{X}}

\newcommand{\hm}{{\mathrm{Hom}}}

\newcommand{\fD}{\mathfrak{X_{s} }}

\newtheorem{theorem}{Theorem}[section]
\newtheorem{lemma}[theorem]{Lemma}
\newtheorem{prop}[theorem]{Proposition}
\theoremstyle{definition}
\newtheorem{definition}[theorem]{Definition}

\theoremstyle{remark}
\newtheorem{remark}[theorem]{Remark}
\numberwithin{equation}{section}
\theoremstyle{plain}

\newtheorem{corollary}[theorem]{Corollary}

\newtheorem*{ack}{Acknowledgement}
\numberwithin{equation}{section}

\newcommand{\secref}[1]{Section~\ref{#1}}
\newcommand{\thmref}[1]{Theorem~\ref{#1}}
\newcommand{\lemref}[1]{Lemma~\ref{#1}}
\newcommand{\remref}[1]{Remark~\ref{#1}}
\newcommand{\propref}[1]{Proposition~\ref{#1}}

\newcommand{\eqnref}[1]{~{\textrm(\ref{#1})}}

\begin{document}
\title[On Conjugation orbits of semisimple  pairs in rank one]{On Conjugation orbits of semisimple  pairs in rank one}
\author[Krishnendu Gongopadhyay  \and Sagar B. Kalane]{Krishnendu Gongopadhyay \and
 Sagar B. Kalane}
\address{Indian Institute of Science Education and Research (IISER) Mohali,
 Knowledge City,  Sector 81, S.A.S. Nagar 140306, Punjab, India}
\email{krishnendug@gmail.com, krishnendu@iisermohali.ac.in}
\address{Indian Institute of Science Education and Research (IISER) Mohali,
 Knowledge City,  Sector 81, S.A.S. Nagar 140306, Punjab, India}
\email{sagark327@gmail.com}
%\thanks{Gongopadhyay acknowledges partial support from SERB MATRICS grant MTR/2017/000355. Kalane acknowledges support from a UGC SRF. }
\date{\today}
\begin{abstract}
We consider the Lie groups $\SU(n,1)$ and $\Sp(n,1)$ that act as isometries of the complex and the quaternionic hyperbolic spaces respectively. We classify pairs of semisimple elements in $\Sp(n,1)$ and $\SU(n,1)$ up to conjugacy.  This gives local parametrization of the representations $\rho$ in $\hm(\F_2, G)/G$ such that both $\rho(x)$ and $\rho(y)$ are semisimple elements in $G$, where $\F_2=\langle x, y\rangle$, $G=\Sp(n,1)$ or $\SU(n,1)$. We use the 
 $\PSp(n,1)$-configuration space $\M(n,i,m-i)$ of ordered $m$-tuples of distinct points in $\overline{\h^n}$, where the first $i$ points in an $m$-tuple are boundary points,  to classify the semisimple pairs.  Further, we also classify points on $\M(n,i,m-i)$. 
\end{abstract}
\maketitle

\section{Introduction}
The rank one symmetric spaces of non-compact types are known to be the real, complex, quaternionic hyperbolic spaces and the Cayley plane.  Classifying pairs of elements up to conjugacy in the isometry group of a rank one symmetric space of non-compact type is a problem of potential interest. This problem is related to the understanding of geometric structures in rank one, as well as invariant theory of linear transformations. In this paper, we shall mostly consider the Lie groups $\SU(n,1)$ and $\Sp(n,1)$ that act as isometries of the complex and the quaternionic hyperbolic spaces respectively. Our aim is to classify pairs of semisimple elements in these groups up to conjugation action. Let $\F_2=\langle x, y \rangle$ be the free group generated by two elements $x$ and $y$,  and let $G$ be a Lie group. The conjugation orbit space is the space $\fX(\F_2, G)=\hm(\F_2, G)/G$, where $G$ acts by conjugation on $\hm(\F_2, G)$. This space is often called the deformation space or the character variety. The subset of this space consisting of the representations $\rho$ such that $\rho(x)$ and $\rho(y)$ are semisimple, is denoted by $\fD(\F_2, G)$. For  $G=\SL(2, \R)$ or $\SL(2, \C)$,   it is well-known from the work of Fricke and Vogt that the group generated by a pair of elements is completely classified up to conjugacy by the traces of the generators and the trace of their product, see \cite{gold2}. The respective  deformation spaces have seen much attention due to their connection with the Teichm\"uller theory and Thurston's vision on Kleinian groups.

 It is not an easy problem to generalize the work of Fricke and Vogt to higher dimensions, or in other rank one isometry groups. A starting point to do this could be the simpler problem that asks for classifying pairs up to diagonal conjugation action. Using invariant theory, there are attempts to classify conjugation orbits of pairs of elements in $\SL(n, \C)$ using polynomials involving traces. The work of Procesi \cite{procesi} has given a set of trace coordinates for classifying conjugation orbits of free group representations into $\GL(n, \C)$. Procesi's coordinate system can be restricted to $\SL(n,\C)$, but a minimal family of such coordinates is known only for lower values of $n$, see \cite{law, law2, do}. Since  $\SU(n,1)$ is a real form of $\SL(n+1, \C)$, the pairs of elements in $\SU(n,1)$ may be associated to certain trace parameters, though may not be minimal. Using the work of Lawton \cite{law}, Will \cite{will2} has obtained a set of minimal trace parameters to classify conjugation orbits of $\F_2$ representations in $\SU(2,1)$, also see Parker \cite{pat}. In an attempt to generalize this work,  Gongopadhyay and Lawton \cite{gl} have classified the polystable pairs (that is, the pairs whose conjugation orbits are closed) in $\SU(3,1)$ using 39 real parameters. At the same time, it has been shown that the real dimension of the \hbox{smallest}  possible system of such real parameters to determine any polystable pair is 30. As evident from \cite{gl}, the complexity of the trace parameters increases with  $n$. An explicit set of trace parameters for pairs in $\SU(n,1)$, $n \geq 4$, is still missing in the literature. 

 Using geometric methods, there are attempts to classify `geometric' pairs in rank one isometry groups, mostly $\SU(n,1)$. Recall that an isometry of the complex or the quaternionic hyperbolic space is called \emph{hyperbolic} if it fixes exactly two points on the boundary.  Parker and Platis \cite{pp},  Falbel \cite{falbel} and Cunha and Gusevskii \cite{cugu2}, independently obtained classifications of the hyperbolic pairs in $\SU(2,1)$. A common idea in these works is to 
associate the congruence classes of fixed points of the hyperbolic pairs to a topological space.  It follows from these works that the traces of the hyperbolic elements along with a point on the respective topological spaces classify the hyperbolic pairs. Parker and Platis applied their result to construct Fenchel-Nielsen parameters on the complex hyperbolic quasi-Fuchsian space. Falbel and Platis \cite{fap} obtained geometric structures of the space constructed by Falbel. In \cite{gp2}, Gongopadhyay and Parsad  have generalized the work of  Parker and Platis to classify generic hyperbolic pairs in $\SU(3,1)$, and then to obtain Fenchel-Nielsen type coordinates on a special component of the $\SU(3,1)$ deformation space of surface group representations.  Recently,  Gongopadhyay and Parsad have given a geometric classification of the conjugation orbits of the hyperbolic pairs in $\SU(n,1)$.   An advantage of the approach in \cite{gp3}  is that the complexity for larger $n$ can be handled successfully to provide a classification in arbitrary dimension.

\medskip Let $\Sp(n,1)$ be the isometry group of the quaternionic hyperbolic space $\h^n$.
In this paper we ask for a classification, up to conjugacy,  of pairs of semisimple elements in $\Sp(n,1)$. In other words, we want to classify elements in the space $\fD(\F_2, \Sp(n,1))$. 
Not much is known about the local structure of this space.  A key obstruction in generalizing the above-mentioned works to pairs in $\Sp(n,1)$ is the lack of conjugacy invariants due to the non-commutativity of the quaternions. Because of this, neither the classical invariant theoretic approach nor the geometric approach has a straight-forward generalization for elements in $\fD(\F_2, \Sp(n,1))$. In this paper, we resolve this difficulty by associating certain spatial invariants, along with linear algebraic invariants available in this setup. This gives us a local parametrization of the space $\fD(\F_2, \Sp(n,1))$. 

 In \cite{gk1}, the authors obtained a parametrization of the $\Sp(2,1)$ conjugation orbits of the pairs of hyperbolic elements of $\h^2$, and applied it to obtain Fenchel-Nielsen type parameters for generic representations of surface groups into $\Sp(2,1)$. The main idea in \cite{gk1} was to handle the eigenvalue classes in $\Sp(2,1)$ and the main result followed from a foliation of the eigenspaces by copies of $\C \P^1$. A point on such a $\C \P^1$ has been called a `projective point'.  It is proved in \cite{gk1} that a pair of hyperbolic elements in $\Sp(2,1)$ is completely determined, up to $\Sp(2,1)$ conjugacy,  by the `real traces', the similarity classes of quaternionic cross ratios, the angular invariants of the fixed points and the projective points.  
 Our main result in this paper generalizes this result in \cite{gk1}, as well as the main result in \cite{gp3}, to classify of a point in $\fD(\F_2, \Sp(n,1)$. As a byproduct, we classify points in $\fD(\F_2, \SU(n,1)$. This gives a system of  local parameters to points in $\fD(\F_2, G)$, where $G=\Sp(n,1)$ or $\SU(n,1)$.  We briefly mention  the key ideas behind our main result that generalizes the above work. 

\medskip Let $\M(n,i,m-i)$ be the space of $\PSp(n,1)$-congruence classes of ordered $m$-tuples of distinct points on $\overline \h^n$, where the first $i$ elements in the $m$-tuples belong to $\partial \h^n$,  and the remaining others are from $\h^n$. Given an element $g$ in $\Sp(n,1)$, it has a representation $g_{\C}$ in $\GL(2n+2, \C)$. The coefficients of the characteristic polynomials of $g_{\C}$ are certain conjugacy invariants of $g$ and the collection of such coefficients is called the  \emph{real trace} of $g$. Along with these real traces, another idea in our approach is to associate tuple of points on $\overline \h^n=\h^n \cup \partial \h^n$ to a semisimple pair $(\rho(x), \rho(y))$, and then project the $\Sp(n,1)$-conjugation orbit of the pair to the moduli space $\M(n,i,m-i)$. Intuitively, the semisimple pairs are seen here as  equivalence classes of `moving frames'.  To each pair, we associate points coming from the closure of the totally geodesic quaternionic lines given by these `frames'.  This association is not well-defined. However, given  a semisimple pair, the orbit of such points under the group action induced by the change of eigenframes gives a well-defined association, and we denote
the space of such orbits by  $\mathcal{QL}_n$. A point on this space that corresponds to a given semisimple pair is called the \emph{canonical orbit} of the pair. This space has a topological structure that comes from the topological structure of the moduli space $\M(n,i,m-i)$. 

 However, the canonical orbits along with the real traces do not give the complete set of invariants that classify the pairs. To complete the classification, we have generalized the main idea of `projective points'  used in \cite{gk1}.  This associates certain spatial invariants to the semisimple  pairs and is a crucial ingredient in the classification.   Let $T$ be a semisimple element in $\Sp(n,1)$.  Let $\lambda \in \H \setminus\R$ be a chosen eigenvalue representative in the similarity class of eigenvalues $[\lambda]$ of $T$ with multiplicity $m$, $m \leq n$, that is, the eigenspace of $[\lambda]$ can be identified with $\H^m$. The $[\lambda]$-eigenspace decomposes into a space of the complex $m$-dimensional subspaces of $\H^m$, that may be identified with the complex  Grassmannian manifold $G_{m, 2m}$.   We call it the \emph{eigenvalue Grassmannian} of $T$ corresponding to the eigenvalue class $[\lambda]$. Each point on this Grassmannian corresponds to an `eigenset' of $[\lambda]$.

\medskip With these notions,  the main result of this paper is the following. 

\begin{theorem}\label{mainth}
Let $(A, B)$ be a semisimple pair in $\Sp(n,1)$ such that $A$ and $B$ do not have a common fixed point. Then a point on the $\Sp(n,1)$ conjugation orbit of $(A, B)$ is determined by the real traces $\tr(A)$, $\tr(B)$, the  canonical orbit of $(A, B)$ on $\mathcal{QL}_n$, and a point on each of the eigenvalue Grassmannians of $A$ and $B$.  
\end{theorem}

An immediate corollary is the following.
\begin{corollary}
Let $\rho$ be an element in $\fD(\F_2, \Sp(n,1))$, $\F_2=\langle x, y \rangle$. Then $\rho$ is determined uniquely by $\tr(\rho(x))$, $\tr(\rho(y))$, the canonical orbit of $(\rho(x), \rho(y))$ on $\mathcal{QL}_n$,  and a point on each of the eigenvalue Grassmannians of $\rho(x)$ and $\rho(y)$. 
\end{corollary}

The methods that we have followed to establish the above results, also carry over to the case of a pair of semisimple elements in $\SU(n,1)$. For elements of $\SU(n,1)$, the underlying hyperbolic space is defined over the complex numbers, and hence there is no ambiguity regarding the conjugacy invariants. The coefficients of the characteristic polynomials serve as well-defined conjugacy invariants for individual elements. Geometric invariants like the cross-ratios are also well-defined. Accordingly, we have the following special case of \thmref{mainth}. This was proved for the hyperbolic pairs in \cite{gp2}. The following is an extension of \cite[Theorem 1.1]{gp2} to semisimple pairs. Here $\Tr(A)$ denote the usual trace of a complex matrix. 

\begin{corollary}
Let $\rho$ be an element in $\fD(\F_2, \SU(n,1))$, $\F_2=\langle x, y \rangle$. Then $\rho$ is determined uniquely by $\Tr(\rho(x)^i)$, $\Tr(\rho(y)^i)$, $1\leq i \leq \lfloor (n+1)/2 \rfloor$, and the  canonical orbit of $(\rho(x), \rho(y))$. 
\end{corollary} 

\medskip     In the above understanding of the conjugation orbits  of the semisimple pairs,  the moduli space of $\PSp(n,1)$-congruence classes of ordered $m$-tuples of distinct points on $\overline \h^n$, $m \geq 4$, is used. As per the above idea, we project an isometry pair onto this space to associate the spatial invariants. In the second part of this paper, we classify points on this space that provides some understanding of its topology. The problem to obtain configuration space of ordered tuples of points on a topological space is a problem of independent interest. The general problem may be stated as follows: 

  Let $X$ be a topological space and $G$ be a group acting diagonally on the ordered $m$-tuples of points on  $X$: for $p=(p_1, \ldots, p_m)$ in $X^m$, $g$ in $G$, 
$$(g, p) \mapsto (gp_1, gp_2, \ldots, gp_m).$$
In general, this is a difficult problem to understand the  orbit space  $X/G$  under this action. However, there are cases when this can be done using the underlying structure of $X$.  A basic example is the case when $X$ is the circle $\s^1$ and one considers ordered quadruple of points on $\s^1$ under the action of the group $\SL(2, \R)$ that acts by the M\"obius transformations on the circle. In this case, the cross ratios of four points essentially determine the orbit space. When $\X$ is the Riemann sphere and $G$ is the group $\SL(2, \C)$ of the M\"obius transformations, similar result also happens. 

 Let $\ch^n$ denote the $n$-dimensional complex hyperbolic space and $\partial \ch^n$ be its boundary. In this case, $X=\ch^n \cup \partial \ch^n$  and $G=\SU(n,1)$. For $k=3$, this problem is related to the classification of the congruence classes of triangles, and it was solved by Cartan, e.g.  \cite{gold}.  The Cartan's angular invariants determine these classes completely.  Another work along this direction was given by Brehm \cite{Brehm} who associated shape invariants to such triples.  The $\SU(n,1)$-congruence classes of ordered tuples of points on $ \ch^n\cup \partial \ch^n$ was obtained by Hakim and Sandler \cite{has}, also see,  Brehm and Et-Taoui \cite{Brehm2}. However, neither of these works gave a complete picture of the moduli space. Cunha and Gusevskii completely solved this problem for $\SU(n,1)$-congruence classes of points on $\partial \ch^n$ and obtained a clear description of the moduli space in \cite{cugu1}.  Gusevskii et. al. also obtained the moduli space of $\SU(n,1)$-congruence classes of points on $\ch^n$ and on the polar space, see \cite{cugu3}, \cite{Niko}.  There are several works on classifications of the $\SU(2,1)$-congruence classes of quadruples of distinct points on $\partial \ch^2$, see \cite{falbel}, \cite{cugu2}, \cite{pp}. All these works are independent of each other and have used different approaches.  

 The above works motivate  the same problem when $X=\h^n \cup \partial \h^n$ and $G=\Sp(n,1)$. The case $k=3,$ in this case, follow from the work of Apanasov and Kim \cite{ak}, who used angular invariants similarly as in the complex hyperbolic case. Recently, the $\Sp(2,1)$-congruence classes of quadruples of points on $\partial \h^n$ has been classified by Cao \cite{Cao}. This classification has been applied by the authors in \cite{gk1}. There have been several recent works to obtain the moduli space of $\Sp(n,1)$-congruence classes of ordered $k$-tuples on $\partial \h^n$, see \cite{Caos},  \cite{gou}. 

\medskip In this paper, we generalize the above works to classify points on $\M(n,i,m-i)$.  We use the same framework following the work of Brehm and Et-Taoui in \cite{bet}, or H\"ofer \cite{hof}, using Gram matrices.  We associate certain numerical invariants to the congruence classes in order to describe the points on $\M(n,i,m-i)$.  Let $p=(p_1, \ldots, p_m)$ be an ordered $m$-tuple of points on $\overline{\hh^n}$ such that first $i$ elements are from $\partial \h^n$ and the remaining ones from $\h^n$. Let $G(p)=(g_{ij})$ be the Gram matrix associated to $p$. We associate the following invariants to $p$. 
 
\medskip \emph{Cross-ratios}: Given an ordered quadruple of pairwise distinct points $(z_1, z_2, z_3, z_4)$ on $\h^n \cup \partial \h^n$, their Kor\'anyi-Reimann quaternionic cross ratio is defined by
$$\X(z_1, z_2, z_3, z_4)=[z_1,z_2,z_3,z_4]={\langle {\bf z}_3, {\bf z}_1 \rangle \langle {\bf z}_3, \bf z_2 \rangle}^{-1} { \langle {\bf z}_4, {\bf z}_2\rangle \langle   {\bf z}_4, {\bf z}_1 \rangle^{-1}},$$
where, for $i=1,2,3,4$,  ${\bf z}_i$ is a  lift of $z_i$. We associate cross ratios to $p=(p_1, \ldots, p_m)$ as follows: 
$$ \X_{1r}=\X(\p_2, \p_1, \p_3, \p_{r}), ~ \X_{2s}=\X(\p_1, \p_2, \p_3, \p_s), ~ \X_{3s}=\X(\p_1, \p_3, \p_2, \p_s),~\X_{ks}=\X(\p_1, \p_k, \p_2, \p_s),$$
for $ (i+1) \leq r \leq m,~ 4 \leq s \leq m,~ 4 \leq k \leq i, ~k<s$.

\medskip A simple count shows that there are a total $d$ number of cross ratios in the above list, where $d=\frac{i(i-3)}{2}+i(m-i)$ with $i$ is the number of null points in $p$. For simplicity of notation, we shall denote them by $(\X_1, \ldots, \X_d)$ unless otherwise required. 

\medskip \emph{Distance Invariants}:
Let $p_i$ and $p_j$ be two distinct negative points in $\hh^n.$ We define distance invariant $d_{ij}$ by:
$d_{ij} = \dfrac{{\langle \p_j, \p_i\rangle}{\langle \p_i,\p_j\rangle}}{{\langle \p_j,\p_j\rangle}{\langle \p_i,\p_i\rangle}}$. The quantity $d_{ij}$ is $\PSp(n,1)$ invariant and it is independent of the chosen lifts of the points. 

\medskip \emph{Angular invariants}: The quaternionic Cartan's angular invariant associated to a triple $(z_1, z_2, z_3)$ on $\h^n \cup \partial \h^n$ is given by the following, see \cite{ak}, \cite{Cao},\cite{Caos}, $$\A(z_1,~z_2,~z_3)=\arccos\frac{\Re(-\langle \z_1, \z_2, \z_3\rangle)}{|\langle \z_1, \z_2, \z_3\rangle|}. $$
where  $\langle \z_1, \z_2, \z_3\rangle=\langle \z_1, \z_2\rangle\langle \z_2, \z_3\rangle\langle \z_3, \z_1\rangle$. We associate angular invariants to $p$ as: $\A_{ij}=\A(\p_1, \p_i, \p_j)$. 

\medskip  
\emph{Rotation invariants}: If $\A_{ij}$ is  non-zero, we further associate a numerical invariant $u_{ij}$  given by: $u_{ij}=\frac{\Im(g_{ij})}{|\Im(g_{ij})|}$. If $\A_{ij}$ is zero, then we shall assume $u_{ij}=0$. The $\Sp(1)$-conjugacy class of $u_{ij}$ is called the \emph{rotation invariant} of $p$. For simplicity of notation, we shall denote them as $u_0, u_1, u_2, \ldots, u_t$, with the understanding that $u_i$ denotes only non-zero rotation invariant for $1 \leq i \leq t$ and $u_0= u_{23}$.  

\medskip With the above notions, we have the following. 
 
\begin{theorem}\label{cong}  A point $[p]$ in $\M(n,i,m-i)$, $p=(p_1, \ldots, p_m)$, is determined completely by the $\Sp(1)$ congruence class of the $(d+t+1)$-tuple 
$$W=(u_{0},u_{1},\ldots,u_{t}, \X_{1},\ldots, \X_{d}), ~ d=\frac{i(i-3)}{2}+i(m-i), ~m \geq 4, ~t=\frac{(m-i)^2-(m-i)}{2}-l,$$  
the angular invariants $\A_{23}$, 
$\A_{i_1j_1}$ and the distance invariants $d_{i_1j_1}$, for $i<i_1,j_1 \leq m$, where $l$ is the number of zero valued rotation invariants $u_{ij}$.  
\end{theorem}

\medskip We have rotation invariants of the cross ratios 
$$\eta_{i}=\frac{\Im(\X_{i})}{|\Im(\X_{i})|}.$$
If $\X_i$ is a real number for some $i$, we assume $\eta_{i}=0$. 

\medskip   Let $p=(p_1, \ldots, p_m)$ be such that $\A(p_1, p_2, p_3) \neq 0$. Note that the $\Sp(1)$-congruence class of the ordered tuple  $F=(u_{0}, u_{1},\ldots, u_{t}, \X_1,\ldots, \X_{d})$ is associated to the $\Sp(1)$-congruence class of the ordered tuple $\mathfrak n=(u_{0}, u_{1}, \ldots, u_{t}, \eta_1, \ldots, \eta_k)$ of points on $\s^2$, where $k$ is the number of non-real cross ratios. Hence the above theorem can be restated in the following form.
\begin{corollary}\label{config} 
 A point $[ p]$ in $\M(n,i,m-i)$, $p=(p_1, \ldots, p_m)$, is determined completely by the angular invariants, the distance invariants,  and a point on the $\Sp(1)$ configuration space of ordered $(k+t+1)$ tuple of points on $\s^2$, where $k$ is the number of  non-real (similarity classes of) cross ratios, and $t+1$ is the number of non-zero angular invariants. 
\end{corollary}

Since $\Sp(1)$ is isomorphic to $\SU(2)$,  the description of the $\Sp(1)$-configuration space of ordered tuples of points can be obtained from the work \cite{bet}.  Restricting to the special cases of the boundary points and the points on $\h^n$ respectively, the above theorem gives the following. 

\begin{corollary}\label{congn} A point $[ p]$ in $\M(n,m,0)$ , $p=(p_1, \ldots, p_m)$, is determined completely by the $\Sp(1)$ congruence class of the $(d+1)$-tuple $F=(u_{0}, \X_1,\X_2,\ldots, \X_{d})$, $d = \frac{m(m-3)}{2}$, $m \geq 4$,  and the angular invariant $\A_{23}$.
\end{corollary}
\begin{corollary}\label{congh}
A point $[ p]$ in $\M(n,0,m)$ , $p=(p_1, \ldots, p_m)$, is determined completely by the angular invariants, distance invariants and $\Sp(1)$ congruence class of the unit pure quaternions associated to $p$. 
\end{corollary}
Let $\M_c(n,i,m-i)$ denote the $\SU(n,1)$-configuration space of ordered tuples of points on $\overline \ch^n$. As an application of the above theorem, we obtain a classification of points on $\M_c(n,i,m-i)$.  Over the complex numbers, conjugacy invariants like the traces and the cross ratios are well-defined. Accordingly, it is much simpler to classify the points on the space $\M_c(n, i, m-i)$. The following corollary is an extension of Theorem 3.1 of Cunha and Gusevskii in \cite{cugu1}. 

\begin{corollary}
A point $[ p]$ in $\M_c(n,i,m-i)$, $p=(p_1, \ldots, p_m)$, is determined completely by
the complex cross ratios
$\X_1,\ldots, \X_{d}, ~ d=\frac{i(i-3)}{2}+i(m-i), ~m \geq 4,$  
the angular invariants $\A_{23}$, $\A_{i_1j_1}$,  and the distance invariants $d_{i_1j_1}$, for $i<i_1,j_1 \leq m$.
\end{corollary} 

Using these results, it is not hard to obtain an explicit description of $\M(n,i, m-i)$ following similar arguments as in \cite{cugu1} for the complex case and \cite{Caos} in the quaternion case, also see \cite{gou}. We omit the details. 
%\medskip If $p=(p_1, \ldots, p_m)$ be the ordered tuple  of points on $\partial \hh^n$, then prove that the %space $\M_0(n,m)$ i.e here $i=0$, has an association with the $\SU(2)$ configuration space of ordered %tuple of points on the two dimensional sphere $\s^2$. In particular, it follows that a point $p=(p_1, \ldots, %p_m)$ on $\M_0(n, m)$, with angular invariant $\A_{23} \neq 0$,  is determined by the real parts and %moduli of  quaternionic cross ratios associated to   $p$,  $\A$ and  a point on the  $\SU(2)$-configuration %space of ordered $(k+1)$-tuples of points on  $\s^2$,  where $k$ is the number of non-real cross-ratio %classes. We deduce this result and the concerning notions in \secref{nullp}. In particular our result shows %that the space $\M_0(n, m)$ has a complicated `fibration' over the space $\R^{d+1} \times (\R-\{0\})^d$, %where $d=\frac{m(m-3)}{2}$, when $p_1, p_2, p_3$ do not belong to a totally real subspace or a $\H$-line. 
\medskip 

\subsubsection*{Structure of the paper} In \secref{prel}, we discuss some preliminary results that will be used later on. The crucial notion of the eigenvalue Grassmannian is given in \secref{grass}. In \secref{semi}, we determine when two semisimple elements are equal.  In \secref{canpo}, we discuss the notion of `associated points' of a semisimple isometry.   The association of `canonical orbits' to the conjugacy class of a pair of semisimple elements has been elaborated in   \secref{corbit}. We prove \thmref{mainth} in \secref{pfmth}. Finally, in \secref{moduli} we prove \thmref{cong}. 

\section{Preliminaries}\label{prel} 

\subsection{The Quaternions}Let $\H$ denote the division ring of quaternions. Let  $\H^{\ast}$ denote the multiplicative group  $\H \setminus\{0\}$ of non-zero quaternions. Recall that every element of $\H$ is of the form  $a_{0}+a_{1}\i+a_{2}\j+a_{3}\k$, where $a_{0},a_{1},a_{2},a_{3}\in \R$, and  $ \i,\j,\k$ satisfy relations:  $\i^{2}=\j^{2}=\k^{2}=\i\j\k=-1$. Any $a\in {\H}$ can be uniquely written as  $a=a_{0}+a_{1}\i+a_{2}\j+a_{3}\k$.
We define $\Re(a)=a_{0}$ = the real part of $a$ and $\Im(a)=a_{1}\i+a_{2}\j+a_{3}\k=$ the imaginary part of $a$. Also, define the conjugate of $a$ by $\overline {a}= \Re(a)-\Im(a)$.
The norm of $a$ is $|a|=\sqrt{a_0^{2}+a_1^{2}+a_2^{2}+a_3^{2}}$.
We identify the subfield $\R + \R \i$ with the standard complex plane $\C$.
Two non-zero quaternions ${a,b}$ are said to be \emph{similar} if there exists a non-zero quaternion $ {c}$ such that $ {b=c^{-1}ac}$ and we write it as $ {a\backsim b}$. It is easy to verify that $ {a \backsim b}$ if and only if $ {\Re(a)=\Re(b)}$ and $|a|=|b|$. Thus the similarity class of every quaternion $a$ contains a pair of complex conjugates with absolute-value $|a|$ and real part equal to $\Re( a)$.  The multiplicative group $\H \setminus\{0\}$ is denoted by $\H^{\ast}$.

\begin{lemma}  Every quaternionic element has a polar coordinate representation. \end{lemma}
\begin{proof}
	Let  $a=a_{0}+a_{1}\i+a_{2}\j + a_{3}\k$ be a quaternion as above. We want to write its polar form. We can have $a=a_{0}+ v$ where $v=a_{1}$i$+a_{2}$j$+a_{3}$k. Now observe that ${\lvert{a}\lvert}^2$ =${{a_0}}^2$+${\lvert{v}\lvert}^2$, where ${\lvert{v}\lvert}=\sqrt{{a_{1}}^2+{a_{2}}^2+{a_{3}}^2 }$. So if $a\neq0$  we get  $1 = {\big(\frac{a_0 }{\lvert a \lvert }}\big)^{2} + {\big(\frac{\lvert v \lvert}{\lvert a \lvert }}\big)^{2}$. Now put $\frac{a_0}{\lvert a \lvert}= \cos(\theta)$ and $\frac{\lvert v \lvert}{\lvert a \lvert}= \sin(\theta)$ where $\theta \in [0,\pi]$. If $ v\neq0$ then we have $a=a_{0}+ v = {\lvert a \lvert}\big (\frac{a_0 }{\lvert a \lvert} + {\frac{\lvert v \lvert}{\lvert a \lvert}}{\frac{v}{\lvert v \lvert}}\big)$. Thus we get $a= {\lvert a \lvert}{(\cos(\theta)+ \eta {\sin(\theta)}})$ where $ \eta = \frac{v}{\lvert v \lvert}\hspace{1mm}  and \hspace{1mm} \theta \in [0,\pi]$. If $v=0$ then $\theta=0$.
\end{proof}

\begin{remark} In Lemma 2.1, $x={\cos(\theta)+ \eta {\sin(\theta)}}$ is a unitary quaternion number. Since $x{\overline{x}}=1={\overline{x}}x .$
\end{remark}

\subsubsection*{Commuting quaternions} Two non-real quaternions commute if and only if their imaginary parts are scaled by a real number, see  \cite[Lemma 1.2.2]{cg} for a proof. Let $Z(\lambda)$ denotes the centralizer of $\lambda \in \H \setminus\R$. Then $Z(\lambda)=\R + \R \lambda$. For some non-zero $\alpha \in \H$,  $Z(\lambda)=\alpha \C \alpha^{-1}$, where $\C=Z(re^{\i \theta})$, $r=|\lambda|$, $\Re \lambda=r \cos \theta$. Given a non-real quaternion $q$, we call $Z(q)$ the \emph{complex line} passing through $q$.
 Note that if $\lambda \in \H\setminus\R$, then $Z(\overline{\lambda}^{-1})={Z(\lambda)}$.

\subsection{Matrices over quaternions}  Let $\V$ be an $n$-dimensional  right vector space over $\H$.  Let $T$ be a right linear transformation of $\V$. Then $T$ is represented by an $n \times n$ matrix over $\H$. Invertible linear maps of $\V$ are represented by invertible $n \times n$ quaternionic matrices. The group of all such linear maps is denoted by $\GL(n, \H)$. For more details on linear algebra over quaternions, see \cite{lr}. In the following, we briefly recall the notions that will be used later on.

 \medskip Let $T \in \GL(n, \H)$ and $v \in \V$, $\lambda \in \H^{\ast}$, are such that $T(v)=v \lambda$, then for $\mu \in \H^{\ast}$ we have 
$$T(v \mu)=(v \mu) \mu^{-1} \lambda \mu.$$ Therefore eigenvalues of $T$ occur in similarity classes and if $v$ is a $\lambda$-eigenvector, then $v \mu \in v \H$ is a $\mu^{-1} \lambda \mu$-eigenvector.
The one-dimensional right subspace spanned by $v$ will be called the \emph{eigenline} to $\lambda$.  Thus the eigenvalues are no more conjugacy invariants for $T$, but the similarity classes of eigenvalues are conjugacy invariant.   Note that each similarity class of eigenvalues contains a unique pair of complex conjugate numbers. We shall choose one of these complex eigenvalues $re^{\i \theta}$, $\theta \in [0, \pi]$,  to be the representative of its similarity class. Often we shall refer them as `eigenvalues', though it should be understood that our reference is towards their similarity classes. In places where we need to distinguish between the similarity class and a representative, we shall denote the similarity class of an eigenvalue representative $\lambda$ by $[\lambda]$.
\subsection{Quaternionic hyperbolic space}
Let $\V=\H^{n,1}$ be the $n$-dimensional right vector space over $\H$ equipped with the Hermitian form of signature $(n,1)$ given by $$\langle\z,\w\rangle=\w^{\ast}H\z=\bar w_{n+1}z_{1}+\bar w_2 z_2+\cdots+ \bar w_{n} z_{n}+\bar w_1 z_{n+1},$$
where $\ast$ denotes conjugate transpose. The matrix of the Hermitian form is given by
\begin{center}
$H=\left[ \begin{array}{cccc}
            0 & 0 & 1\\
           0 & I_{n-1} & 0 \\
 1 & 0 & 0\\
          \end{array}\right],$
\end{center}
where $I_{n-1}$ is the identity matrix of rank $n-1$.
We consider the following subspaces of $\H^{n,1}:$
$$\V_{-}=\{\z\in\H^{n,1}:\langle\z,\z \rangle<0\}, ~ \V_+=\{\z\in\H^{n,1}:\langle\z,\z \rangle>0\},$$
$$\V_{0}=\{\z-\{{\bf 0}\}\in\H^{n,1}:\langle\z,\z \rangle=0\}.$$
A vector $\z$ in $\H^{n,1}$ is called \emph{positive, negative}  or \emph{null}  depending on whether $\z$ belongs to $\V_+$,   $\V_-$ or  $\V_0$. Let $\P:\H^{n,1}-\{0\}\longrightarrow  \H \P^n$ be the right projection onto the quaternionic projective space. Image of a vector $\z$ will be denoted by $z$.  The quaternionic hyperbolic space $\h^n$ is defined to be $\P \V_{-}$. The ideal boundary $\partial\h^n$ is defined to be $\P \V_{0}$. So we can write $\h^n=\P(\V_{-})$ as
$$\h^n=\{(w_1,\ldots, w_n)\in\H^n \ : \ 2\Re(w_1)+|w_2|^2+\cdots+|w_n|^2<0\},$$
where for a point $\z=\begin{bmatrix}z_1 & z_2 & \ldots & z_{n+1}\end{bmatrix}^T \in \V_- \cup \V_0$, $w_i=z_i z_{n+1}^{-1}$ for $i=1, \ldots, n$. This is the Siegel domain model of $\h^n$. Similarly one can define the ball model by replacing $H$ with an equivalent Hermitian form $H'$ given by the diagonal matrix: $H'=diag(-1,1, \ldots, 1)$. 
We shall mostly use the Siegel domain model here.

There are two distinguished points in $\V_{0}$ which we denote by  $\bf{o}$ and $\bf\infty,$ given by
$$\bf{o}=\left[\begin{array}{c}
               0\\0\\ \vdots \\ 1\\
              \end{array}\right],
~~ \infty=\left[\begin{array}{c}
               1\\0 \\ \vdots \\0\\
              \end{array}\right].$$\\
Then we can write $\partial\h^n=\P(\V_{0})$ as
$$\partial\h^n-\infty=\{(z_1,\ldots,z_n)\in\H^n:2\Re(z_1)+|z_2|^2+\cdots+|z_n|^2=0\}.$$\\
Note that $\overline{\h^n}=\h^n \cup \partial \h^n$. 

 Given a point $z$ of $\overline{\h^n}-\{\infty\} \subset\H \P^n$ we may lift $z=(z_1,\ldots, z_n)$ to a point $\z$ in $\V$, called the \emph{standard lift} of $z$. It is represented in projective coordinates by
 $$\z=\left[\begin{array}{c}
                z_1\\ \vdots \\ z_n\\1\\
               \end{array}\right].$$
 The Bergman metric in $\h^n$ is defined in terms of the Hermitian form given by:
$${ds}^2=-\frac{4}{\langle \z,\z \rangle^2} \det \left[\begin{array}{cc}
                                                    \langle \z,\z \rangle & \langle d\z ,\z \rangle\\
                                                    \langle \z,d\z \rangle & \langle d\z,d\z \rangle\\
                                                   \end{array}\right].$$
If $z$ and $w$ in $\h^n$ correspond to vectors $\z$ and $\w$ in $\V_{-}$, then the Bergman metric is also given by the distance $\rho$:
$$\cosh^2\bigg(\frac{\rho(z,w)}{2}\bigg)=\frac{\langle\z,\w\rangle \langle\w,\z\rangle}{\langle\z,\z\rangle \langle\w,\w\rangle}.$$

More information on the basic formalism of the quaternionic hyperbolic space may be found in \cite{cg}, \cite{kip}. 

\subsection{Isometries}   \label{ltr}
 Let ${\rm Sp}(n,1)$ be the isometry group of  the Hermitian form $\langle .,.\rangle$. Each matrix $A$ in ${\rm Sp}(n,1)$ satisfies the relation $A^{-1}=
H^{-1}A^{\ast}H$, where $A^{\ast}$ is the conjugate transpose of $A$. The isometry group of  $\h^n$ is the projective unitary group ${\rm PSp}(n,1)={\rm Sp }(n,1)/\{\pm I\}$. However, we shall mostly deal with $\Sp(n,1).$

\medskip Based on their fixed points, isometries of $\h^n$ are classified as follows:
\begin{enumerate}
\item An isometry is \emph{elliptic} if it fixes a point on $\h^n$.
\item An isometry is \emph{parabolic} if it fixes exactly one point on $\partial\h^n$.
\item An isometry is \emph{hyperbolic} if it fixes exactly two points on $\partial\h^n$.
\end{enumerate}
%\begin{definition}  If all the eigenvalues of $g$ are non-reals, we shall call it %\emph{loxodromic}.  \end{definition}
The elliptic and hyperbolic isometries are semisimple. Now we define the following terminology for describing conjugacy classification of semisimple isometries. Let $g$ be a semisimple element in $\Sp(n,1)$. Let $\lambda$ be an eigenvalue of $g$, counted without multiplicities. Then $\lambda$ is called negative, resp. null, resp. positive if the corresponding $\lambda$-eigenvector is negative, resp. null, resp. positive. Accordingly, a similarity class of eigenvalues is negative, null or positive according to its representative is negative, null, or positive, respectively. For details about conjugacy classification of the isometries we refer to Chen-Greenberg \cite{cg}. We note the following fact that is useful for our purposes.

\begin{lemma}\label{hycl}{\rm (Chen-Greenberg )\cite{cg}}
\begin{enumerate} 
\item 
 Two elliptic elements in ${\rm Sp}(n,1)$ are conjugate if and only if they have the same negative class of eigenvalues, and the same positive classes of eigenvalues.

\item Two hyperbolic elements in ${\rm Sp}(n,1)$ are conjugate if and only if they have the same similarity classes of eigenvalues.
\end{enumerate} 
\end{lemma}
\begin{lemma}\label{emb}
The group $\Sp(n,1)$ can be embedded in the group ${\rm GL}(2n+2, \C)$.
\end{lemma}
\begin{proof}
Write $\H=\C\oplus{\bf j}  \C$. For $A\in \Sp(n,1)$, express $A=A_1+{\bf j}A_2$,
{where} $ A_1, A_2\in M_{n+1}(\C)$. This gives an embedding $A \mapsto A_{\C}$ of $\Sp(n,1)$ into ${\rm GL}(2n+2, \C)$, where
\begin{equation}\label{crep} A_{\C}= \left(
                          \begin{array}{cc}
                            A_1 &  -\overline{A_2}\\
                          {A_2}  & \overline{A_1} \\
                          \end{array}
                        \right).
\end{equation}
\end{proof}
The following lemma follows from the above. 
\begin{prop}\label{rt}
Let $A$ be an element in $\Sp(n,1)$. Let $A_{\C}$ be the
corresponding element in ${\rm GL}(2n+2, \C)$. The characteristic polynomial
of $A_{\C}$ is of the form
\begin{equation*}\chi_A(x)=\sum_{j=0}^{2n+2} a_j x^{2(n+1)-j},\end{equation*}
where $a_0=1=a_{2n+2}$ and for $1 \leq j \leq n+1$, $a_j=a_{2(n+1)-j}$. 
If $A$ is hyperbolic, then the conjugacy class of  $A$ is determined by the real numbers $a_j$, $1 \leq j \leq n+1$. If $A$ is elliptic then the conjugacy class of  $A$ is 
determined by the real numbers $a_j$, $1 \leq j \leq n+1$, along with the negative-type eigenvalue of $A$. 
\end{prop}

\begin{definition}
Let $A$ be a semisimple element in $\Sp(n,1)$. The real $n$-tuple $(a_1, \ldots, a_n)$ as in  \propref{rt} will be called the  \emph{real trace} of $A$ and we shall denote it by $tr_{\R}(A)$.
\end{definition}
\subsection{The Cross Ratios}\label{crai}
Given an ordered quadruple of pairwise distinct points $(z_1, z_2, z_3, z_4)$ on $\overline{\h^n}$, their Kor\'anyi-Reimann (quaternionic) cross ratio is defined by
$$\X(z_1, z_2, z_3, z_4)=[z_1,z_2,z_3,z_4]={\langle {\bf z}_3, {\bf z}_1 \rangle \langle {\bf z}_3, \bf z_2 \rangle}^{-1} { \langle {\bf z}_4, {\bf z}_2\rangle \langle   {\bf z}_4, {\bf z}_1 \rangle^{-1}},$$
where, for $i=1,2,3,4$,  ${\bf z}_i$, are lifts of $z_i$. Unlike the complex case, quaternionic cross ratios are not independent of the chosen lifts. However, similarity classes of the cross ratios are independent of the chosen lifts. In other words, the conjugacy invariants obtained from the cross ratios are $\Re(\X)$ and $|\X|$. Under the action of the symmetric group ${\rm S}_4$ on a tuple, there are exactly three orbits, see \cite[Prop 3.1]{platis}. This implies that the moduli and real parts of the quaternionic cross ratios are determined by the following three cross ratios.
$$\X_1(z_1, z_2, z_3, z_4)=\X(z_1, z_2, z_3, z_4)=[\z_1, \z_2, \z_3, \z_4],$$
$$\X_2(z_1, z_2, z_3, z_4)=\X(z_1, z_4, z_3, z_2)=[\z_1, \z_4, \z_3, \z_2],$$
 $$\X_3(z_1, z_2, z_3, z_4)=\X(z_2, z_4, z_3, z_1)=[\z_2, \z_4, \z_3, \z_1].$$
 Usually, cross ratios are defined for boundary points but we can generalize it for $\overline{\h^n}$ and we will use it in \secref{moduli}.
Platis defined cross ratios for boundary points and proved that these cross ratios satisfy the following  real relations: 
$$|\X_2|=|\X_1||\X_3|, ~ \hbox{and, }~~  2|\X_1|\Re(\X_3)\geq |\X_1|^2+|\X_2|^2-2\Re(\X_1)-2\Re(\X_2)+1,$$
where equality is attained if and only if certain conditions hold, see \cite[Proposition 3.4]{platis}. 
 W. Cao also defined it for $\overline{\h^n}$, for more details, see \cite[Section 3]{Cao}. 
\subsection{Cartan's angular invariant}\label{cai}
Let $z_1,~z_2,~z_3$ be three distinct points of $\overline{ \h^n}=\h^n \cup \partial\h^n$, with lifts $\z_1,~\z_2$ and $\z_3$ respectively. The quaternionic Cartan's angular invariant associated to the triple $(z_1, z_2, z_3)$ was defined by Apanasov and Kim in \cite{ak} and is given by the following:
$$\A(z_1,~z_2,~z_3)=\arccos\frac{\Re(-\langle \z_1, \z_2, \z_3\rangle)}{|\langle \z_1, \z_2, \z_3\rangle|}.$$
The angular invariant is an element in $[0, \frac{\pi}{2}]$. It is independent of the chosen lifts and also $\Sp(n,1)$-invariant. The following proposition shows that this invariant determines
any triple of distinct points on $\partial\h^n$ up to ${\rm Sp }(n,1)$-equivalence. For a proof see  \cite{ak}.
\begin{prop}\label{cai1}{\rm \cite{ak}}
Let $z_1,~z_2,~z_3$ and $z_1^\prime,~z_2^\prime,~z_3^\prime$ be triples of distinct points of $\partial\h^n$. Then $\A(z_1,~z_2,~z_3)=\A(z_1^\prime,~
z_2^\prime,~z_3^\prime)$ if and only if there exist $A\in {\rm Sp }(n,1)$ so that $A(z_j)=z_j^\prime$ for $j=1,2,3$.
\end{prop}
Further, it is proved in \cite{Cao} that $(z_1, z_2, z_3)$ lies on the boundary of an $\H$-line, resp. a totally real subspace, if and only if $\A=\frac{\pi}{2}$, resp.  $\A=0$.

%\subsection{Brehm's Shape Invariant} Similar to the Angular invariant, we have Brehm's shape invariant %\cite{Brehms} for triple of negative points on $\h^n$.  
%\begin{definition} let $(p_1,p_2,p_3)$ be triple of negative points in  $\hh^n$ such that $\parallel %p_1\parallel=1 $ then shape invariant $\sigma$ is defined by $\sigma = \Re(\langle p_1,p_2\rangle \langle %p_2, p_3\rangle \langle p_3,p_1 \rangle)$, where \\ $\parallel p_1\parallel=\parallel p_2\parallel = \parallel %p_3\parallel = 1$.
%\end{definition}

\section{ Eigenvalue Grassmannians}\label{grass}
Let $T$ be an invertible semisimple matrix over $\H$ of rank $n$. Let $\lambda \in \H \setminus\R$ be a chosen eigenvalue of $T$ in the similarity class $[\lambda]$ with multiplicity $m$, $m \leq n$. Thus, the eigenspace of $[\lambda]$ can be identified with $\H^m$. Let $\lambda$ be a representative of $[\lambda]$.  Consider the $\lambda$-\emph{eigenset}: $S_{\lambda}=\{ x \in \H^n \ | \ Tx =x \lambda \}$. Note that this set can be identified with the subspace $ Z(\lambda)^m$ in $\H^m$. Thus each eigenset of a $[\lambda]$-representative is a copy of ${\C}^m$ in $\H^m$.  So, the set of eigensets in the $[\lambda]$-eigenspace can be identified with the set of $m$ dimensional complex subspaces of $\H^m$. By identifying $\H^m$ to $\C^{2m}$, we obtain the set of $[\lambda]$-eigensets as the space complex $m$-dimensional subspaces of $\C^{2m}$, which is the complex  Grassmannian manifold $G_{m, 2m}$, or simply $G_m$. This is a compact connected smooth complex manifold of complex dimension $m^2$. We call it the \emph{eigenvalue Grassmannian} of $T$ corresponding to the eigenvalue $[\lambda]$, or simply $[\lambda]$-Grassmannian. Each point on this Grassmannian corresponds to an eigenset of $[\lambda]$.  When $m=1$, the eigenvalue Grassmannian is simply $\C \P^1$, and a point on such $\C \P^1$ has been termed as a \emph{projective point} in \cite{gk1}. The $[\lambda]$-eigenvalue Grassmannian is a conjugacy invariant of an eigenvalues but individual points are not. They help us to distinguish individual isometries in the same conjugacy class.

\section{Semisimple Isometries in $\Sp(n,1)$} \label{semi}
 In $\Sp(n,1)$, the semisimple isometries are classified as hyperbolic and elliptic. 
\subsubsection{Elliptic Isometries} 
Let $A$ be an elliptic element in $\Sp(n,1)$.  Recall that an eigenvalue of an elliptic element $A$ always has norm $1$.  Let $\lambda$ be an eigenvalue from the similarity class of eigenvalues $[\lambda]$ of $A$. Let $\x$ be a $\lambda$-eigenvector. Then $\x$ defines a point $x$  on $\H\P^n$, that is either a point on  $\h^n$ or a point in $\P(\V_{+})$. The lift of $x$ in $\H^{n,1}$ is the quaternionic line $\x \H$. We call $x$ as \emph{projective fixed point} of $A$ corresponding to $[\lambda]$. 

\medskip Let the eigenvalues of $A$ be the $n+1$ unit complex numbers $e^{i \theta_1}, \ldots, e^{i \theta_{n+1}}$, where $e^{i \theta_1}$ is negative and $e^{i \theta_k}$, $k=2, \ldots, n+1$, positive. Up to conjugacy, $A$ is of the form:

\begin{equation}\label{li1}
  E_A(\theta_1, \theta_2, \ldots, \theta_{n+1})=\begin{bmatrix}
                         e^{\i\theta_1} &  0& 0 & \ldots & 0 & 0 \\
                         0 & e^{\i\theta_2} & 0 & \ldots & 0 & 0 \\
                         & & \ddots & & & \\
                         0 & 0 & 0 &\ldots & e^{\i \theta_{n}}& 0 \\
                         0& 0& 0 &\ldots & 0 &  e^{\i\theta_{n+1}}\\
                        \end{bmatrix}.\\
\end{equation}

 Let $C_A=\left[\begin{array}{ccccc} \x_{1, A} & \x_{2,A} & \ldots & \x_{n+1, A}\\
  \end{array}\right]$ be the matrix corresponding to the eigenvectors of the above eigenvalue representatives. We can choose $C_A$ to be an element of $\Sp(n,1)$ by normalizing the eigenvectors:
$$\langle \x_{1,A}, \x_{1,A} \rangle=-1,~\langle \x_{j,A}, \x_{j,A} \rangle=1, ~ j \neq 1.$$
 Then $A=C_A E_A C_A^{-1}$. 
\subsubsection{Eigenspace decomposition of an elliptic element} \label{edec} 
Suppose $A$ is an elliptic element in $\Sp(n,1)$. Suppose that the eigenvalue classes  of $A$ are represented by  $e^{\i \theta_1}$,  $e^{ \i \theta_2}, \ldots, e^{\i \theta_k}$, ordered so that $e^{\i \theta_1}$ is the negative eigenvalue. Let $\V_{\theta_i}$ be the eigenspace to the eigenvalue class of $e^{\i \theta_i}$. Let $m_i=\dim \V_{\theta_i}$. We call $(m_1, \ldots, m_k)$ the \emph{multiplicity} of $A$. 
The space  $\H^{n,1}$ has the following orthogonal decomposition  into eigenspaces (here $\oplus$ denotes orthogonal sum):
\begin{equation} \label{decom 2} \H^{n,1}=\V_{\theta_1} \oplus \ldots \oplus \V_{\theta_k},\end{equation} 
Change of eigenbasis amounts to conjugation by an element $C$, and $CAC^{-1}=A$ if and only if $C \in Z(A)$. So, a normalised eigenbasis of $A$ is determined up to  conjugation action of  $Z(A)=\prod_{i=1}^k Z(A|_{\V_{\theta_i}})$ on each of the summands. 

For description of the centralizers, see \cite{g1}. Let $e^{\i \theta_j}$ represent an eigenvalue of $A$ with multiplicity $m_j$. It follows from \cite{g1} that $Z(A|_{\V_{\theta_j}})$ can be identified with $\U(m_j-1,1)$ if the eigenvalue is negative, and with $\U(m_j)$ otherwise. So, 
if $A$ does not have eigenvalues $1$ or $-1$,  given the multiplicity $(m_1, \ldots, m_k)$, $Z(A)$ may be identified to the group
$$Z(A)=\U(m_1-1,1) \times\U(m_2) \times \ldots \U(m_k).$$
When $A$ has an eigenvalue $1$ or $-1$, one of the factors in the above product is replaced by $\Sp(m_1-1, 1)$ or $\Sp(m_i)$ depending upon the eigenvalue is negative or positive. 
\subsubsection{Hyperbolic isometries}\label{hi} Let $A$ be a hyperbolic element in $\Sp(n,1)$. Let $\lambda$ be an eigenvalue from the similarity class of eigenvalues $[\lambda]$ of $A$. Let $\x$ be a $\lambda$-eigenvector. Then $\x$ defines a point $x$ on $\H\P^n$, that is either a point on  $\partial \h^n$ or a point in $\P(\V_{+})$. The lift of $x$ in $\H^{n,1}$ is the quaternionic line $\x \H$. Then $x$ is a \emph{projective fixed point} of $A$. The quaternionic line $\x \H$ is the \emph{eigenline} spanned by the eigenvector $\x$. 

 \medskip 
There are two eigenvalue classes of null-type and the respective eigenlines correspond to attracting and repelling fixed points. Let $r_A \in \partial\h^n$ be the \emph{ repelling fixed point } of $A$ that corresponds to the eigenvalue $re^{\i \theta}$ and let $a_A$ be the \emph{attracting fixed point} corresponding to the eigenvalue $r^{-1}e^{\i \theta}$.  Let $r_A$ and $a_A$ lift to eigenvectors $\r_A$ and $\a_A$ respectively.  Let $\x_{j,A}$ be an eigenvector corresponding to $e^{\i \phi_j}$. We may further assume that $\theta$, $\phi_j$ are in $[0, \pi]$. The point $x_{j,A}$ on $\P(\V_+)$ is the \emph{polar-point} of $A$. For $(r,\theta, \phi_1, \ldots, \phi_{n-1})$ as above,  let $ E_A(r, \theta, \phi_1, \ldots, \phi_{n-1})$, or simply $E_A$ be the matrix: 

\begin{equation}\label{li2}
  E_A(r, \theta, \phi_1, \ldots, \phi_{n-1})=\begin{bmatrix}
                         re^{\i\theta} &  0& 0 & \ldots & 0 & 0 \\
                         0 & e^{\i\phi_1} & 0 & \ldots & 0 & 0 \\
                         & & \ddots & & & \\
                         0 & 0 & 0 &\ldots & e^{\i \phi_{n-1}}& 0 \\
                         0& 0& 0 &\ldots & 0 &  r^{-1}e^{\i\theta}\\
                        \end{bmatrix}.\\
\end{equation}
 Let $C_A=\left[\begin{array}{ccccc} \a_A & \x_{1,A} & \ldots & \x_{n-1, A}&\r_A\\
  \end{array}\right]$ be the matrix corresponding to the eigenvectors of the above eigenvalue representatives. We can choose $C_A$ to be an element of $\Sp(n,1)$ by normalizing the eigenvectors:
$$\langle \a_A, \r_A \rangle=1,~\langle \x_{j,A}, \x_{j,A} \rangle=1.$$
 Then $A=C_A E_A C_A^{-1}$. 
\subsubsection{Eigenspace decomposition of a hyperbolic element} \label{eidlox} Suppose $A$ is a hyperbolic  element in $\Sp(n,1)$. Suppose also that the eigenvalue classes  are represented by  $re^{\i \theta}$, $r^{-1} e^{\i \theta}$, $r>1$, and $e^{ \i \phi_1}, \ldots, e^{\i \phi_k}$. Let $m_i =\dim \V_{\phi_i}$.  We call $m=(m_1, \ldots, m_k)$ the \emph{multiplicity} of $A$. Then $\H^{n,1}$ has the following orthogonal decomposition  into eigenspaces (here $\oplus$ denotes orthogonal sum):
\begin{equation} \label{decom 1} \H^{n,1}=\L_r \oplus \V_{\phi_1} \oplus \ldots \oplus \V_{\phi_k},\end{equation} 
where $\L_r$ is the $(1,1)$ (right) subspace of $\H^{n,1}$ spanned by $\a_A$, $\r_A$.  As in the elliptic case, change of normalised eigenbasis of $A$ is determined up to the conjugation action of $Z(A)$.

\subsection{Determination of the semisimple elements} In the following, we determine the semisimple isometries. We shall associate certain spatial  parameters to an isometry that would determine it completely. Proposition \ref{lox} below will be crucial for classification of the conjugation orbits of semisimple pairs.

\begin{definition} Let $A$ and $A'$ be two semisimple elements with the same set of eigenvalue classes. We shall say that $A$ and $A'$ have the same projective fixed points if for each non-real class $[\lambda]$ they have the same projective fixed point $x_{\lambda}$. 
\end{definition} 

\begin{definition}
Let $A$ and $A'$ be two semisimple elements having a common non-real eigenvalue class $[\lambda]$. We shall say that $A$ and $A'$ have the same point on the $[\lambda]$-eigenvalue Grassmannian if they have the same point on the eigenvalue Grassmannian with respect to the representative $\lambda$ of $[\lambda]$. 
\end{definition}
\begin{lemma}\label{lox}
Let $A$ and $A'$ be two semisimple elements in $\Sp(n,1)$. Then $A=A'$ if and only if they have the same real trace, the same projective fixed points and the same point on each of the eigenvalue Grassmannians.
\end{lemma}
\begin{proof}
If $A=A'$, then the statement is clear. For the converse, let $A$ and $A'$ be two semisimple elements of $\Sp(n,1)$. Since they have the same real traces, they have the same eigenvalue classes. Further, $A$ and $A'$ have the same projective fixed points, hence $\x_{j,A'}=\x_{j,A} q_j$ for $1 \leq j \leq n+1$, and $q_j \in \H^*$. Hence,
$$A=C_AE_A C_A^{-1}, ~~\hbox{ and } A'=C_{A'} E_A C_{A'}^{-1}.$$

  Let
$$D=\begin{bmatrix} q_1 & 0 & 0 & \ldots & 0  \\ 0 & q_3 & 0& \ldots & 0 \\&  & \ddots & & \\ 0 & 0 & \ldots & q_{n-1} & 0 \\ 0 & 0 & \ldots& 0 & q_2 \end{bmatrix}.$$
Then $A=C_A E_A C_A^{-1}$ and, $A'=C_A D E_A D^{-1} C_A^{-1}$. Thus $A=A'$ if and only if $D$ commutes with $E_A$, which is equivalent to the condition of having  the same point on each of the eigenvalue Grassmannians. 
\end{proof}

In linear algebraic terms, the above lemma may be re-stated as follows. 
\begin{corollary}
Let $A$ and $A'$ be two semisimple elements in $\Sp(n,1)$. Then $A=A'$ if and only if the following holds. 
\begin{enumerate}
\item $A$ and $A'$ have the same similarity classes of eigenvalues.
\item To each eigenvalue class $[\lambda]$, $A$ and $A'$ have the same eigenspace, and 
\item to each representative $\lambda'$ of $[\lambda]$, $A$ and $A'$ have the same eigensets. \end{enumerate} 
\end{corollary} 
\section{Associated points of an isometry}\label{canpo}
\subsection{Associated points of a hyperbolic element} \label{canpoh}

\begin{definition} 
Let $A$ be a hyperbolic element in $\Sp(n,1)$. Let $\mathcal B_A=\{\a_A, \x_{1, A}, \ldots, \x_{n-1, A}, \r_A\}$ be an ordered set of eigenbasis corresponding to $A$, normalised  so that for $1 \leq l \leq n-1$,
\begin{equation} \label{n1}
\langle \a_A,\r_A \rangle=\langle \x_{l,A},\x_{l,A}\rangle =1, ~\langle \x_{l, A}, \x_{m, A}\rangle=0, ~l \neq m.
\end{equation}
  Define a set of $n+1$ boundary points associated to $A$ as follows:
\begin{equation}\label{b1}
\p_{1,A}=\a_A,~ \p_{2, A}=\r_A,~\p_{l,A}=(\a_A-\r_A)/\sqrt{2}+\x_{l-2,A}, ~3\leq l\leq n+1.
\end{equation}
The set  $p_A=\big\{{p_{1, A}, \ldots, p_{n+1, A}}\big\}$ is called a set of \emph{associated points} of $A$. \end{definition} 

\begin{lemma}\label{ash} Let $A$ be a hyperbolic element of $\Sp(n,1)$. Let $Z(A)$ denote the centralizer of $A$ in $\Sp(n,1)$. Then 
the associated points of $A$ is well-defined up to an orbit of the subgroup $Z(A)$. 
\end{lemma} 
\begin{proof}
Let $A$ be a hyperbolic element in $\Sp(n,1)$. Let $\mathfrak p=(p_{1. A}, \ldots, p_{n+1, A})$ be a tuple of associated points of $A$ given by an eigenbasis $\mathcal B_A$ as above. If we choose another normalised  eigenbasis $\mathcal B'_A$ of $A$, we get another set of associated points $\mathfrak p'$. 
The map $M$ changing $\mathcal B_A$ to $\mathcal B'_A$ satisfies $MAM^{-1}=A$. Thus, $M$ belongs to $Z(A)$, and $M(\mathfrak p)=\mathfrak p'$. 
\end{proof} 

\subsection{Associated points of  an elliptic element}\label{canpoe}
\begin{definition} Let $A$ be an elliptic element in $\Sp(n,1)$. Let $\mathcal B_A=\{ \x_{1, A}, \ldots, \x_{n+1, A}\}$ be a set of eigenvectors of $A$ chosen so that for $1 \leq l \leq n+1$,
\begin{equation} \label{m1}
\langle \x_{1,A},\x_{1,_A} \rangle=-1, ~\langle \x_{j,A},\x_{j,A}\rangle =1, ~j \neq 1, ~ \langle \x_{l, A}, \x_{m, A}\rangle=0, ~l \neq m.
\end{equation}
 Define a set of $n+1$ points on $\h^n$ as follows:
\begin{equation}\label{b2}
\p_{1,A}=\x_{1,A},~\p_{l,A}= \x_{1, A} \sqrt 2+\x_{j,A}, ~2\leq j \leq n+1.
\end{equation}
The set $p_A=\big\{{p_{1, A}, \ldots, p_{n+1, A}}\big\}$ is called a set of \emph{associated points} to $A$. \end{definition} 

\begin{lemma} Let $A$ be an elliptic element of $\Sp(n,1)$. Let $Z(A)$ denote the centralizer of $A$ in $\Sp(n,1)$. Then 
the associated points of $A$ is well-defined up to an orbit of the subgroup $Z(A)$. 
\end{lemma}

The proof of the above lemma is similar to the proof of \lemref{ash}. 

\begin{definition}
Given a semisimple element $A$ in $\Sp(n,1)$, a set of eigenbasis of the type $\mathcal B_A$ as given above, will be called an \emph{eigenframe} of $A$. 

A set of $n+1$ vectors of $\H^{n,1}$ alike an eigenframe will be called an \emph{orthonormal frame}. 
\end{definition} 

\subsubsection{Change of associated points amounts to a change of eigenbases}
The following lemma is easy to prove and it shows that the change of associated points amount change of eigenbases. 
\begin{lemma}\label{nor}
 Let $A,~A^{'}$ be semisimple elements in $ \Sp(n,1)$ with  chosen eigenframes. Let $p_A=\big\{{p_{1, A}, \ldots, p_{n+1,A}}\big\}$ and $p_{A'}=\big\{{p_{1, A'}, \ldots, p_{n+1, A'}}\big\}$ be sets of associated points to $A$ and $A'$, respectively. Suppose that there exists $C\in \Sp(n,1)$ such that $C(p_{l,A})=p_{l,A'}$,  $1\leq l \leq n$. Then $C(x_{j,A})=x_{j,A'}$ for all $j$. 
\end{lemma}
\begin{proof}
We prove our claim for hyperbolic isometries. The elliptic case is similar. 

Let $C(\p_{l,A})=\p_{l,A'}\alpha_l$ for $1\leq l\leq n$. Observe that $\langle \p_{1,A},\p_{l,A} \rangle=-1/\sqrt{2}=\langle \p_{1,A'},\p_{l,A'} \rangle$
   and $\langle \p_{2,A},\p_{l,A} \rangle=1/\sqrt{2}=\langle \p_{2,A'},\p_{l,A'} \rangle \text{ for }3\leq l\leq n$. Since $C\in \Sp(n,1)$ preserve the form $\langle .,. \rangle$, from these relations we have 
 $ \alpha_{i}={ \bar \alpha_1}^{-1}={\bar \alpha_{2}}^{-1}$ for $ 3\leq i \leq n$. Now, $\langle \p_{1, A}, \p_{2, A}\rangle=1$ gives $|\alpha_1|=1$. Hence,  $C(x_{i-2,A})=x_{i-2,A'}$ for 
  $3\leq i \leq n$. This implies $C(x_{n-1, A})=x_{n-1, A'}$. 
  \end{proof}

\section{Canonical orbit  of a pair} \label{corbit}
\subsection{Canonical orbit  of a pair of hyperbolics}\label{corbith}
\subsubsection{Moduli of normalised boundary points} Consider the set  $\mathcal E$ of ordered tuples of boundary and polar points on $(\partial \h^n)^4 \times \P(\V_+)^{2(n-1)}$ given by a pair of orthonormal frames $(F_1, F_2)$:  
$$\mathfrak p=(q_1, q_2, r_1, r_2, \ldots, r_{n-1}, q_{n+1}, q_{n+2}, r_{n+1}, \ldots, r_{2n-1}).$$
 This corresponds to pair of orthonormal frames of $\H^{n,1}$: 
$$\hat {\mathfrak p}=(\q_1, \q_2, \r_1, \r_2, \ldots, \r_{n-1}, \q_{n+1}, \q_{n+2}, \r_{n+1}, \ldots, \r_{2n-1}),$$
where $\{ \q_1, \q_2\} \cap \{ \q_{n+1}, \q_{n+2}\} =\emptyset$,  $\langle \q_i, \q_i \rangle=0=\langle \q_{n+i}, \q_{n+i} \rangle$ for  $i=1, 2$, 
$\langle \r_j, \r_j\rangle=\langle \r_{n+j}, \r_{n+j} \rangle=1$ for all $j=1, \ldots, n-1$,
$\langle \q_1, \q_2 \rangle=\langle \q_{n+1}, \q_{n+2} \rangle=\langle \q_1, \q_{n+2} \rangle=1$, 
$\langle \q_i, \r_j \rangle=0=\langle \q_{n+i}, \r_{n+j} \rangle$,  for $i=1, 2$,  $j=1, \ldots, n-1$. 

\medskip To each such point, we have an ordered tuple of boundary points, not necessarily distinct,  $(p_1, \ldots, p_{2n+2})$ satisfying the conditions:

\begin{equation}\label{eq1} 
\langle \p_1, \p_2 \rangle= \langle \p_{n+2}, \p_{n+3}\rangle=\langle \p_1, \p_{n+2}\rangle=1,\end{equation}

\begin{equation}\label{eq12}
\langle \p_i, \p_j \rangle=-1=\langle \p_{n+i}, \p_{n+j} \rangle, i \neq j, ~i, j=3, \ldots, n-1;
\end{equation}

\begin{equation} \label{eq2} \langle \p_1, \p_i \rangle=-\frac{1}{\sqrt 2}=\langle \p_{n+2}, \p_{n+i}\rangle,~ i=3, \ldots, n-1;\end{equation}
\begin{equation} \label{eq3}  \langle \p_2, \p_k \rangle=\frac{1}{\sqrt 2}=\langle \p_{n+2}, \p_{n+k} \rangle, ~k=3, \ldots, n-1,\end{equation}
where $\p_i$ denotes the standard lift of $p_i$ for each $i$.  Note that $\p_i, \p_{n+i}$, $i=3, \ldots, n$, may not be distinct. In this case, we relabel them and write them as a ordered tuple of distinct boundary points $\hat {\mathfrak p}=(p_1, p_2, \ldots, p_t)$, $n+3 \leq t \leq 2n+2$, so that they correspond to the original ordering of $\mathfrak p$. 

\medskip Let $\mathcal L_t$ be the section of $\M(n, t,0)$ defined by the equations \eqnref{eq1}--\eqnref{eq3}, and the ordering as described above. Let $\mathcal L$ be the disjoint union $\mathcal L=\bigcup_{t=n+3}^{ 2n+2} \mathcal L_t$.

\subsubsection{Canonical orbit of a hyperbolic pair} 
Let $(A, B)$ be a hyperbolic pair in $\Sp(n,1)$.  Given the pair $(A, B)$, consider the ordered set of eigenvectors of $A, B$ given by the tuple
$$\mathfrak e=(\a_A,  \r_A, \x_{1, A}, \ldots, \x_{n-1, A}, \a_B, \r_B, \x_{1, B}, \ldots, \x_{n-1, B}), $$
with normalization as follows:
\begin{equation} \label{n1'}
\langle \a_A,\r_A \rangle=\langle \x_{i,A},\x_{i,A}\rangle =1, ~\langle \x_{i, A}, \x_{j, A}\rangle=0, ~i \neq j.
\end{equation}
\begin{equation} \label{n2}
\langle \a_B,\r_B \rangle=\langle \x_{i,B},\x_{i,B}\rangle =1, ~\langle \x_{i, B}, \x_{j, B}\rangle=0, ~i \neq j.
\end{equation}
\begin{equation}\label{n3}
\langle \r_A,\a_B \rangle=1.
\end{equation}
  Assign the associated boundary points to $\mathfrak e$ defined by \eqnref{b1}: 
\begin{equation}\label{ep} \mathfrak p=(a_A, r_A,  q_{1, A}, \ldots, q_{n-1, A}, a_B, r_B, q_{1, B}, \ldots, q_{n-1, B}). \end{equation}
 If we change $\mathfrak e$ to another pair of eigenframes $\mathfrak e'$ of $(A, B)$, say 
$$\mathfrak e'=(C(\a_A), C(\r_A),  C(\x_{1, A}), \ldots, C(\x_{n-1, A}), D(\a_B), D(\r_B), D(\x_{1, B}), \ldots, D(\x_{n-1, B})),$$
then since we are not changing $(A, B)$, must have  $C \in Z(A)$ and $D\in Z(B)$. Accordingly, there is an action of $Z(A) \times Z(B)$ by the change of eigenframes, and the point $\mathfrak p$ changes to a point $\mathfrak p'$ on some $\M(n, t, 0)$, where $n+3 \leq t \leq 2n+2$. 
Thus, $\mathfrak e$, and hence $\mathfrak p$ is determined by $(A, B)$ up to the above action of the group $Z(A)\times Z(B)$ on $\mathfrak p$. 

The $Z(A) \times  Z(B)$ action on $\mathfrak p$ above defines a set of points on $\mathcal L$. We shall call this set as a $Z(A) \times Z(B)$ orbit on $\mathcal L$ and denote it by $[\mathfrak p]$. We call $[\mathfrak p]$ the \emph{canonical orbit} of $(A, B)$. The association of the canonical orbit $[\mathfrak p]$   to the conjugacy class of $(A, B)$ is well-defined.

\medskip It follows from the description of centralizers in \cite{g1} that for all pairs of hyperbolic elements $(A, B)$ with multiplicities $(a_1, \ldots, a_k; b_1, \ldots, b_l)$, and without an eigenvalue $1$ or $-1$, we can identify their centralizers. This induces an action of $Z(A) \times Z(B)$ 
 on $\mathcal L$ by the above construction. The orbit space on $\mathcal L$ under this $Z(A) \times Z(B)$ action is denoted by $\mathcal{QL}_n(a_1, \ldots, a_k; b_1, \ldots, b_l)$. If either of the hyperbolic elements in the pair has an eigenvalue $1$ or $-1$, then the group $Z(A) \times Z(B)$ changes, but the same construction go through. Taking disjoint union of all such orbit spaces, we get a space  $\mathcal{QL}_n$. Each point on    $\mathcal{QL}_n$ corresponds to a conjugacy class of a hyperbolic pair $(A, B)$.

\subsection{Canonical orbit  of a pair of elliptics}\label{corbite}
\subsubsection{Moduli of normalised points} Consider the set  $\mathcal E$ of ordered tuples of points on $(\h^n)^2\cup \P(\V_+)^{2n}$ given by a pair of orthonormal frames $(F_1, F_2)$:  
$$\p=(x_1, \ldots, x_{n+1}, x_{n+2}, \ldots, x_{2n+2}).$$

\medskip To each such point, we have an ordered tuple of  negative points, not necessarily distinct,  
$(p_1, \ldots, p_{2n+2})$ satisfying the conditions:

\begin{equation}\label{eq1'}
\langle \p_{1, A}, \p_{1, A}\rangle=-1=\langle \p_{n+2}, \p_{n+2} \rangle
\end{equation}
\begin{equation}\label{eq2'}
\langle \p_{j, A}, \p_{j, A}\rangle=-1=\langle\p_{j, A'}, \p_{j, A'}\rangle
\end{equation}
\begin{equation}\label{eq3'}
\langle \p_{s, A}, \p_{t, A}\rangle=0=\langle\p_{s, A'}, \p_{t, A'}\rangle, ~ s \neq t, 1 \leq s, t \leq n+1, 
\end{equation}
\begin{equation}\label{eq4'}
\langle \p_{1, A}, \p_{j, A}\rangle=-\sqrt 2=\langle\p_{1, A'}, \p_{j, A'}\rangle, ~ j \neq 1,
\end{equation}
\begin{equation}\label{eq5'}
\langle \p_{l, A}, \p_{m, A}\rangle=- 2=\langle\p_{l, A'}, \p_{m, A'}\rangle, ~ l, m \neq 1,
\end{equation}
where $\p_i$ denotes the standard lift of $p_i$ for each $i$.  Note that $\p_i, \p_{n+i}$, $i=3, \ldots, n$, may not be distinct. If they are not distinct, we relabel them and write them as an ordered tuple of distinct negative points $\hat {\p}=(p_1, p_2, \ldots, p_t)$, $n+3 \leq t \leq 2n+2$, so that they correspond to the original ordering of $\p$. 

\medskip Let $\mathcal L_t$ be the section of $\M(n, 0, t)$ defined by the equations \eqnref{eq1'}--\eqnref{eq5'}, and let also the ordering as described above. Let $\mathcal L$ be the disjoint union $\mathcal L=\bigcup_{t=n+3}^{ 2n+2} \mathcal L_t$.

\subsubsection{Canonical orbit of an elliptic pair} 
Let $A$ and $B$ are elliptic elements in $\Sp(n,1)$ without a common fixed point.  Given the pair $(A, B)$, fix eigenframes of $A$ and $B$  so that
 $$\langle \x_{1,A},\x_{1,A}\rangle=-1=\langle \x_{1,B},\x_{1,B}\rangle, ~\langle \x_{i,A},\x_{i,A}\rangle =\langle \x_{i,B},\x_{i,B}\rangle=1=\langle \x_{1, A}, \x_{1, B}\rangle, 1\leq i \leq n+1. $$
Consider the ordered tuple of eigenvectors $\mathcal B=( \x_{1, A}, \ldots, \x_{n+1, A},  \x_{1, B}, \ldots, \x_{n+1, B})$. 
This gives an ordered tuple of points in $\h^n$ given by
$$\mathfrak p=(p_{1,A}, \ldots, p_{n+1,A},  p_{1,B}, \ldots, p_{n+1,B}),$$
where $\p_{i,A}$, $\p_{i,B}$ are defined by \eqnref{b2}. 

\medskip 
The tuple $\mathfrak p$ is determined by $(A, B)$ up to the above action of the group 
$G=Z(A)\times Z(B)$ on $\mathfrak p$. So, to each pair $(A, B)$, we have a $Z(A)\times Z(B)$ orbit of associated points.  The $Z(A) \times  Z(B)$ action on $\mathfrak p$  defines a set of points on $\mathcal L$. We shall call this set as a $Z(A) \times Z(B)$ orbit on $\mathcal L$ and denote it by $[\mathfrak p]$. We call $[\mathfrak p]$ the \emph{canonical orbit} of $(A, B)$.  The canonical orbit $[\mathfrak p]$  corresponds uniquely to the conjugacy class of $(A, B)$.

\medskip  The above action of $Z(A) \times Z(B)$ on $\mathfrak p$ induces an action of $Z(A) \times Z(B)$ on $\mathcal L$ similarly as described in the previus section. This  gives a $Z(A) \times Z(B)$ orbit  $[\mathfrak p]$ in $\mathcal L$ and we call it the \emph{canonical orbit} of $(A, B)$.  The orbit space on $\mathcal L$ under the above $G$-action will be denoted by the same symbol as in the previous section, $\mathcal{QL}_n(a_1, \ldots, a_k; b_1, \ldots, b_l)$. Taking disjoint union of all such orbit spaces, we get a space  $\mathcal{QL}_n$. Each point on    $\mathcal{QL}_n$ corresponds to a conjugacy class of an elliptic pair $(A, B)$. 
\subsection{Canonical orbit of a mixed pair} 
In this case, the construction is very similar to the elliptic and hyperbolic case. Let $A$ be hyperbolic and $B$  be elliptic  in $\Sp(n,1)$.  Given the pair $(A, B)$, fix a pair of associated orthonormal frames $\mathcal B=(\mathcal B_A, \mathcal B_B)$ so that  the eigenvectors are normalised as in \secref{canpoh} and \secref{canpoe}.  Next we choose an ordering as in the previous section and associate an ordered tuple of points $\mathfrak p$ on $\M(n, n+1, n+1)$.
The $Z(A) \times Z(B)$ action gives a orbit $[\mathfrak p]$ on $\M(n, n+1, n+1)$ as earlier. Taking disjoint union of all such orbits, we get a space, still denoted by $\mathcal {QL}_n$ as in the previous section. Each point on $\mathcal{QL}_n$ corresponds to a conjugacy class of a mixed pair $(A, B)$. 

%Suppose $(A', B')$ is an element in the conjugation orbit of $(A, B)$.   Then$(A', B')$ and $(A, B)$ have the same real traces and the same canonical orbits. Further, $(A', B')$ is determined by a point on each of the eigenvalue Grassmannians of $A$ and $B$. 

\section{Proof of \thmref{mainth}}\label{pfmth} 
\begin{proof}For simplicity, we shall assume that neither $A$ nor $B$ has an eigenvalue $1$ or $-1$. The proof is just similar in these omitted cases. 

\medskip Suppose that $(A, B)$ and $(A', B')$ are hyperbolic pairs having equal real trace and the same canonical orbit. Equality of real traces implies that they have the same  multiplicities,  say $(a_1, \ldots, a_k, b_1, \ldots, b_l)$. 
Following the notation in \secref{semi}, we may assume $A=C_{A} E_{A}C_{A}^{-1},~B=C_{B} E_{B}C_{B}^{-1}$ and similarly for $A'$ and $B'$. In this case, $C_A$ is an element in the subgroup 
$\Sp(1,1) \times \Sp(a_1) \times \ldots \times \Sp(a_k)$:
$$C_A=\begin{bmatrix} \a_A &  L_1 &\ldots& L_k & \r_A \end{bmatrix},$$
where $L_i=\begin{bmatrix} \x_{t_i, A} & \ldots & \x_{t_i +a_i -1, A} \end{bmatrix}$, $E_A$ is the diagonal matrix
$$E_A=\begin{bmatrix} re^{\i \theta} & 0 & & \ldots&  \\ 0 &  I_{a_1}\lambda_1 & 0 & 0 \ldots & 0\\& & 
\ddots & & \\0 & 0 & 0 &  I_{a_k}\lambda_k & 0 \\ 0 & 0 & 0 & 0 & r^{-1} e^{\i \theta} \end{bmatrix},$$

where $I_s$ denotes the identity matrix of rank $s$. Similarly $E_B$ is a diagonal matrix 

$$E_B=\begin{bmatrix} se^{\i \alpha} & 0 & & \ldots&  \\ 0 &  I_{b_1}\mu_1 & 0 & 0 \ldots & 0\\& & 
\ddots & & \\0 & 0 & 0 &  I_{b_l}\mu_l & 0 \\ 0 & 0 & 0 & 0 & s^{-1} e^{\i \alpha} \end{bmatrix},$$
and, 
$$C_B=\begin{bmatrix} \a_B &  K_1 &\ldots& K_l & \r_B \end{bmatrix},$$
where $K_i=\begin{bmatrix} \x_{t_j, B} & \ldots & \x_{t_j +b_j -1, B} \end{bmatrix}$. In the above notation,  $t_i=\sum_{p=1}^{i} a_{p-1}$, $s_j=\sum_{p=1}^{j} b_{p-1}$, $a_0=b_0=1$.

\medskip  Since the canonical orbits are equal, by \lemref{nor} it follows that there exists a $C\in{\rm{{\rm Sp }}}(n,1)$ such that $C(a_{A})=a_{A'},~C(r_{A})=r_{A'}, ~C(a_{B})=a_{B'},~C(r_B)=r_B$, and for $1 \leq i \leq k$, $ 1\leq j \leq l$, 
$$C(\x_{t_i, A}, \ldots, \x_{t_i+a_i-1, A})= M(\x_{t_i, A'}, \ldots, \x_{t_i+a_i-1, A'}),$$
$$C(\x_{t_j, B}, \ldots, \x_{t_j+b_j-1, B})= N(\x_{t_j, B'}, \ldots, \x_{t_j+b_j-1, B'}) ,$$
where $M \in Z(A')$, $N \in Z(B')$. 
Since $A'$ commutes with $M$, $A'M(\x_{t_i, A'})=MA'(\x_{t_i, A'})=M(\x_{t_i, A'}) \lambda_i$. From the above, we also have $M(\x_{t_i, A'})=C(\x_{t_i, A})$, which implies that $CAC^{-1}$ and $A'$ have the same projective fixed point given by $M(x_{t_i, A'})=C(x_{t_i, A})$. 

Thus $CAC^{-1}$ and $A'$ have the same projective fixed points. 
Since $A$ and $A'$ define the same point on each of the eigenvalue Grassmannians and have the same real traces, by \lemref{lox},  $CAC^{-1}=A'$. Similarly, $B'=CBC^{-1}$. 

\medskip Suppose $(A, B)$ and $(A', B')$ are elliptic pairs with the same real traces and the same canonical orbit. Since they have the same traces, their multiplicities are also the same, say $(a_1, \ldots, a_k, b_1, \ldots, b_l)$. In this case, $C_A$ is an element in the subgroup 
$\Sp(a_1-1,1) \times \Sp(a_2) \times \ldots \times \Sp(a_k)$:
$$C_A=\begin{bmatrix}  E_1 & E_2 &  \ldots & E_k  \end{bmatrix},$$
where $E_i=\begin{bmatrix} \x_{t_i, A} & \ldots & \x_{t_i +a_i -1, A} \end{bmatrix}$, and $E_A$ is the diagonal matrix
$$E_A=\begin{bmatrix} \lambda_1 I_{a_1} &  & \\& 
\ddots & & \\&  &  & \lambda_k I_{a_k} \end{bmatrix},$$
where $I_s$ denotes the identity matrix of rank $s$. Similarly for $C_B$, let 
$$E_B=\begin{bmatrix} \mu_1 I_{a_1} &  & \\& 
\ddots & & \\&  &  & \mu_k I_{a_k} \end{bmatrix},$$
and 
$$C_B=\begin{bmatrix}  E'_1 & E'_2 &  \ldots & E'_k  \end{bmatrix},$$
where $E'_i=\begin{bmatrix} \x_{t_i, B'} & \ldots & \x_{t_i +a_i -1, B'} \end{bmatrix}$. 

 Since the canonical orbits are equal, by \lemref{nor} it follows that there exists a $C\in{\rm{{\rm Sp }}}(n,1)$ such that  for $1 \leq i \leq k$, $ 1\leq j \leq l$, 
$$C(\x_{t_i, A}, \ldots, \x_{t_i+a_i-1, A})= M(\x_{t_i, A'}, \ldots, \x_{t_i+a_i-1, A'}) ,$$
$$C(\x_{t_j, B}, \ldots, \x_{t_j+b_j-1, B})= N(\x_{t_j, B'}, \ldots, \x_{t_j+a_i-1, B'}) ,$$
where $M\in Z(A')$, $N \in Z(B')$.  Now, using arguments as in the hyperbolic case, it follows that $CAC^{-1}$ and $A'$ have the same projective fixed points. Hence by \lemref{lox}, $CAC^{-1}=A'$. Similarly, $CBC^{-1}=B'$. 

 Suppose $(A, B)$ and $(A', B')$ are mixed pairs such that $A$, $A'$ are hyperbolic and $B$, $B'$ are elliptic. For the mixed pairs case the argument is similar. This completes the proof. 
\end{proof}

\section{Moduli space of $\PSp(n,1)$ congruence classes of distinct points from $\overline{ \h^n}$}\label{moduli}

\subsection{The Gram Matrix}  A common feature in the  works \cite{cugu1}, \cite{Caos}, and the one formulated in this section is the use of the Gram matrices. This is motivated by the ideas of Brehm and Et-Taoui  in \cite{bet} and H\"ofer in \cite{hof}.  

\medskip Let $p = (p_1,p_2,\ldots,p_m)$ be an ordered $m$-tuple of distinct points in $\overline{ \h^n}$. The Gram matrix associated to $p$ is the matrix $G = (g_{ij}),$ where $g_{ij} = \langle \p_j,\p_i \rangle$ with $\mathfrak p=(\p_1,\p_2,\ldots,\p_m)$ is 
a chosen lift of $p$.

\medskip In this section, without loss of generality, we will assume that first $i$  elements from $p = (p_1,p_2,\ldots,p_m)$ are null and remaining other elements are negative, for $3\leq i\leq m$. The following proposition is in \cite[Prop.1.1]{Cao}.

\begin{definition}
We say that two Hermitian  $m\times m$ matrices $G$ and $K$ are equivalent if there exist non-singular diagonal matrix $D$ such that $G=D^{*}KD$.
\end{definition}

\begin{prop}\label{Cao mixed}
If $\z,\w \in \H^{n,1}-\{0\}$ with $\langle \z,\z \rangle \leq 0$ and $\langle \w,\w \rangle =0$ then either $\w=\z \lambda$ for some $\lambda \in \H$ or $\langle\z,\w \rangle \neq 0.$
\end{prop}

From this Proposition \ref{Cao mixed} we can see that $\langle \z,\w \rangle \neq 0$, for  $z \in \partial\h^n$ and $w \in \h^n$. Also we will have $\langle \z,\w \rangle \neq 0$, for $z \neq w$ together with the condition that either $z,w \in \partial\h^n$ or $z,w \in \h^n$.
Now by using these observations along with an appropriate chosen lift of $p= (p_1,p_2,\ldots,p_m)$,  the  following lemma follows using a similar argument as in \cite{cugu1}.
\begin{lemma}\label{semii}
 Let $p = (p_1,p_2,\ldots,p_m)$ be a $m$-tuple of distinct points in $\overline{ \h^n}$. Then the  equivalence class of Gram matrices associated to $p$ contains a  matrix $G=(g_{kj})$ with $g_{kk}= 0 $ for $k = 1,2,\ldots,i$,  $\lvert g_{23}\rvert =1$, $g_{1j} = 1$  for $j= 2,3,\ldots,i$, $g_{kk}= -1 $ for $k = i+1,i+2,\ldots,m$, and $g_{1k}= r_{1k}$ for $k= i+1,i+2,\ldots,m$, where $r_{1k}$ are real positive numbers.
\end{lemma}
\begin{proof} 
Let $\mathfrak p=(\p_1,\p_2,\ldots,\p_m)$ be a lift of $p$. We want to choose a lift of $p$ such that it will satisfy required conditions of the lemma. As $p_i$'s are distinct points in $\overline{ \h^n}$, we get $\langle{\p_k},\p_k\rangle=g_{kk}=r_k$, where $r_k=0$ for $k = 1,2,\ldots,i$ and $r_k$ is a negative real number for $k = i+1,i+2,\ldots,m$. Now we can find out scalars $\lambda_k$ such that $\langle{\p_k}{\lambda_k},\p_k{\lambda_k}\rangle= 0 $ for $k = 1,2,\ldots,i$ and $\langle{\p_k}{\lambda_k},\p_k{\lambda_k}\rangle= -1 $ for $k = i+1,i+2,\ldots,m$. In particular we can take $\lambda_k = 1$ for $k = 1,2,\ldots,i$ and $\lambda_k = {{\sqrt{{-r_k}}}}~^{-1}$ for $k = i+1,i+2,\ldots,m$. 

Now again rescale $\mathfrak p$ by $\beta_1=1$, $\beta_k = \overline{{\langle{\p_1 \lambda_1},\p_k\lambda_k\rangle}~^{-1}}$ for $k = 2,3,\ldots,i$ and $\beta_k = x_{1k}$ for $k = i+1,i+2,\ldots,m$, where $x_{1k}$ is unitary part of $\langle{\p_1}\lambda_1,\p_k\lambda_k\rangle$. Thus we get conditions $g_{1j} = 1$ for $j= 2,3,\ldots,i$ and $g_{1j}= r_{1k}$ for $k= i+1,i+2,\ldots, m$, where $r_{1k} =\lvert \langle{\p_1 \lambda_1},\p_k \lambda_k\rangle\rvert $ are real positive numbers. Finally for getting condition $\lvert g_{23}\rvert = 1$ we will further rescale $\mathfrak p$ by $\gamma_1= \sqrt{r_{23}}$, $\gamma_k= \frac{1}{\sqrt{r_{23}}}$ for $k = 2,3,\ldots,i$ and $\gamma_k=1$ for $k= i+1,i+2,\ldots, m$, where $r_{23} =\lvert \langle{\p_2\lambda_2\beta_2},\p_3\lambda_3\beta_3\rangle\rvert $. So appropriate lift $\mathfrak p=(\p_1 \lambda_1 \beta_1 \gamma_1,\p_2 \lambda_2 \beta_2 \gamma_2,\ldots,\p_m \lambda_m \beta_m \gamma_m)$ of $p$ proves the lemma. 
\end{proof}

\begin{remark} The Gram matrix $G(\mathfrak p)=(g_{kj})$ in the above lemma \ref{semii} has the form 
\begin{align}\label{3.2}
G(\mathfrak p)=
\left[\begin{array}{ccccc|c}
0&1&1&\cdots &1&\\
1&0&g_{23}&\cdots&g_{2i}&\\
1&\overline{g_{23}}&0&\cdots&g_{3i}&G^*\\
\vdots&\vdots&\vdots&\ddots&\vdots&\\
1&\overline{g_{2i}}&\overline{g_{3i}}&\cdots&0&\\
\hline
&&&&&\\
&&\overline{G}^*&&&A
\end{array}\right],
\end{align}
where $\lvert g_{23}\rvert =1$ , $G^*$ is $i \times (m-i)-$ matrix with each nonzero entry together with having first row contains real positive number and $A$ is $(m-i) \times (m-i)-$ matrix with the form,
\begin{center}
	$A=\left[ \begin{array}{cccc}
	 -1 & g_{(i+1)(i+2)} & g_{(i+1)(i+3)} \cdots & g_{(i+1)m} \\
  \overline{g_{(i+1)(i+2)}} & -1 & g_{(i+2)(i+3)} \cdots & g_{(i+2)m} \\
  
  \vdots  & \vdots  & \vdots  \ddots & \vdots  \\
 \overline{g_{(i+1)(m)}}& \overline{g_{(i+2)m}} & \overline{g_{(i+3)m}} \cdots & -1 \\
	\end{array}\right].$
\end{center}
\end{remark}

 Now observe that the spanning set of $\mathfrak p=(\p_1,\p_2,\ldots,\p_m)$ contains at least one negative point even if $m$-tuple contains all null points. So the spanning set of $\mathfrak p=(\p_1,\p_2,\ldots,\p_m)$  is non-degenerate.

\medskip 
The following proposition follows using similar arguments as in the proof of \cite[Theorem 1]{hof}. It is essentially a consequence of the Witt's extension theorem. 
\begin{prop} \label{gram matrix equivalent} 
Let $p =(p_1,p_2,\ldots,p_m)$ and  $q =(q_1,q_2,\ldots,q_m)$  be the $m$-tuple of distinct points in  $\overline{\hh^n}$. Then $p$ and $q$ are congruent in  $\PSp(n,1)$ if and only if  their associated Gram matrices are equivalent.
\end{prop}
 
\subsection{Semi-normalised Gram matrix}
\begin{definition}
We will call the matrix in lemma \ref{semii} as  \emph{semi-normalised Gram matrix} with respect to lift $\mathfrak p=(\p_1,\p_2,\ldots,\p_m)$ of $p$.
\end{definition}

The following lemma shows that semi-normalised Gram matrix is just an equivalence class.
\begin{lemma}\label{4.1o}  Suppose that the Gram matrix $G(\mathfrak p)$ is semi-normalised Gram matrix for $p$ with respect to lift $\mathfrak p=(\p_1,\p_2,\ldots,\p_m)$. Then $G({\mathfrak p}')$  is still a semi-normalised Gram matrix with $\mathfrak p'=(\p_1{\lambda_1},\ldots,\p_m{\lambda_m})$ if and only if $\lambda_1=\lambda_2=\ldots=\lambda_m$ and $ {\lambda_i}$ $ \in \Sp(1)$.
\end{lemma}
\begin{proof}It follows from ${\langle \p_1\lambda_1,\p_k{\lambda_k}\rangle}=1$ that $\overline{\lambda_k}{\lambda_1} = 1$, for $k= 2,3,
\ldots,i$ as ${\langle \p_1,\p_k\rangle}=1$, and from $\lvert{{\langle \p_2{\lambda_2},\p_3{\lambda_3}\rangle}}\rvert = 1$ that
$\lvert {\overline{\lambda_3}}\rvert{ \lvert \lambda_2 \rvert }=1 $ as $\lvert \langle \p_2,\p_3\rangle \rvert =1$. Thus we have $\vert\lambda_1 \rvert = 1 $ so $\lambda_1 \in \Sp(1)$. So by $\overline{\lambda_k}{\lambda_1} = 1$ for $k= 2,3,\ldots,i$ we have $\lambda_1=\lambda_2=\ldots=\lambda_i$ and $ {\lambda_i}$ $ \in \Sp(1)$.

Also we can see that ${\langle \p_k{\lambda_k},\p_k{\lambda_k}\rangle}=-1$, for $k={i+1},i+2,\ldots,m$ gives that  $\lvert{{\lambda_k\lvert}}=1$ for $k={i+1},i+2,\ldots,m$. Since ${\langle \p_1{\lambda_{1}},\p_j{\lambda_j}\rangle}=r'_{{1}j}$, where $r'_{{1}j}$ are positive real numbers for $j= i+1,i+2,\ldots,m.$ Thus we have $ \overline{ {\lambda_{j}}}~r_{1j}~{\lambda_{1}}=r'_{1j}$, where $\langle \p_1,\p_j\rangle=r_{1j}$. By using the fact $\lvert{{\lambda_i\lvert}}=1$ implies $r'_{1j}=r_{1j}$ for $j= i+1,i+2,\ldots,m.$ As we can commute real numbers with quaternions we get $\overline{ {\lambda_{j}}}{\lambda_{1}}=1$ for $j= i+1,i+2,\ldots,m.$ Thus we have $\lambda_j=\lambda_{1}$ for $j= i+1,i+2,\ldots,m$ with $\lvert{{\lambda_{1}\lvert}}=1$ i.e., $\lambda_{1} \in \Sp(1)$.

Conversely, we can verify that if $\mathfrak p'=(\p_1{\lambda_1},\p_2{\lambda_1},\ldots,\p_m{\lambda_1})$ with $\lvert{{\lambda_1\lvert}}=1$, then $G({\mathfrak p}')$ is matrix of the form  \eqnref{3.2}.
\end{proof}

\begin{remark}\label{rk1} We can represent semi-normalised Gram matrix $G(\mathfrak p)=(g_{kj})$  by\\ $V_G=(r_{1(i+1)},r_{1(i+2)},\ldots,r_{1m},g_{23},g_{24},\ldots,g_{2m},g_{34},\ldots,g_{3m},\ldots,g_{{m-1}m})$ in ${\mathbb{H}}^{t}$, where $\lvert g_{23}\rvert = 1$ and $t=\frac{(m^2-m-2i+2)}{2}$. Action of  $\Sp(1)/\lbrace1,-1\rbrace$ on ${\mathbb{H}}^{t}$ by
  $(\mu, V_G)$ $\rightarrow$ $\overline{\mu}{V_G}\mu$ gives the orbit  $O_{V_G}=\lbrace {\overline{\mu}{V_G} \mu}:\forall \mu \in \Sp(1) \rbrace$.
\end{remark}
\begin{lemma}\label{4.3o} Let $G_1$ and $G_2$ be two semi-normalised Gram matrices represented by $V_{G_1}$ and $V_{G_2}$ resp. Then $G_1$ and $G_2$ are equivalent if and only if $O_{V_{G_1}}=O_{V_{G_2}}$.
\end{lemma}
\begin{proof} Let  $G_1$ and $G_2$ be two semi-normalised Gram matrices of $p$ and $q$ respectively, where $p$ and $q$ are m-tuples of distinct points in $\overline{ \h^n}$  with lifts $\mathfrak p=(\p_1,\p_2,\ldots,\p_m)$ and $\mathfrak q=(\q_1,\q_2,\ldots,\q_m)$ respectively. The equivalence of $G_1$ and $G_2$ implies that there exist non-singular diagonal matrix $D$ such that $G_1=D^{*}G_{2}D$ where  $D=diag(\lambda_1,\ldots,\lambda_m)$. So we have 
\begin{equation}
G_1=(g_{ij})=(\langle \p_j,\p_i\rangle)=(\overline{\lambda_i}\langle \q_j,\q_i\rangle\lambda_j)=(\overline{\lambda_i}{g_{ij}}'\lambda_j)=(\langle \q_j{\lambda_j},\q_i{\lambda_i}\rangle).
\end{equation}\label{x}
Lemma \ref{4.1o} above gives now $\lambda_1=\lambda_2=\ldots=\lambda_m =\mu$ and 
$\mu \in \Sp(1)$.
So from equation \ref{x}, $g_{ij}=\overline{\mu}{g_{ij}}'\mu$, where $\mu$ $ \in \Sp(1)$. So we get $O_{V_{G_1}}=O_{V_{G_2}}$.\\
Conversely, if $O_{V_{G_1}}=O_{V_{G_2}}$, we want to find non-singular diagonal matrix $D$ such that $G_1=D^{*}G_{2}D$, where $D=diag(\lambda_1,\lambda_2,\lambda_3,\ldots,\lambda_m)$. As $V_{G_1}$ lies in $O_{V_{G_1}}$, ${V_{G_1}}={\overline{\mu}{V_{G_2}}\mu}$ for some $\mu$ $ \in \Sp(1)$ and $g_{ij}=\overline{\mu}{g_{ij}}'\mu$. Thus we get $D=(\mu,\mu,\ldots,\mu)$, where $\mu$ $ \in \Sp(1)$.
 \end{proof}
 
\begin{lemma}\label{ov} Let $p$ and $q$ be m-tuples of distinct points in $\overline{ \h^n}$. Then $p$ and $q$ are congruent in  $\PSp(n,1)$ if and only if $O_{V_{G_1}}=O_{V_{G_2}}$, where $V_{G_1}$ and $V_{G_2}$ are represented by semi-normalised Gram matrices associated to $p$ and $q$ respectively.
\end{lemma}
\begin{proof}
By Proposition \ref{gram matrix equivalent}, $p$ and $q$ are congruent in  $\PSp(n,1)$ if and only if their associated Gram matrices are equivalent. Let $G_1$ and $G_2$ be the two semi-normalised Gram matrices associated to $p$ and $q$ respectively. We represent them by $V_{G_1}$ and $V_{G_2}$ respectively. Now by using \ref{4.3o}, we get the result $O_{V_{G_1}}=O_{V_{G_2}}$.	
\end{proof}

\subsection{Configuration space of ordered tuples of points}
\medskip Note that $\X_{1j}= g_{23}r_{1j}{g_{2j}}^{-1}$, $\X_{2j}=\overline{g_{23}} g_{2j}  r_{1j}^{-1}$, $\X_{3j}=\overline{g_{23}}^{-1}g_{3j} r_{1j}^{-1}$, $\X_{kj}=\overline{g_{2k}}^{-1}  g_{kj} r_{1j}^{-1}$, where $r_{1j}=1$ if $2\leq j\leq i$. Hence the Gram matrix $G(\p)=(g_{ij})$ can be read off from these invariants.

\medskip Let $p_1,~p_i,~p_j$ be three distinct points of $\overline{ \h^n}=\h^n \cup \partial\h^n$, with lifts $\p_1,~\p_i$ and $\p_j$, respectively. We can write $\langle \p_1, \p_j, \p_i\rangle = {|\langle \p_1, \p_j, \p_i\rangle|}($cos$ \theta_{ij} + u_{ij}~$sin$ \theta_{ij}) = |\langle \p_1, \p_j, \p_i\rangle| e^{u_{ij}\theta_{ij}}$, where $u_{ij} \in \s^2$ is a unit pure quaternion. If $\langle \p_1, \p_j, \p_i\rangle$ is a real number then $u_{ij}$ is undefined, and we shall assume in such cases that$u_{ij}=0$.  Note that 
 $$ \A(p_1,~p_i,~p_j) = \arg (\langle \p_1, \p_j, \p_i\rangle) = \arg (r_{1j}g_{ij}r_{1i}) = \arg(g_{ij})=\theta_{ij}.$$ 
 It follows from \remref{rk1} that the $\Sp(1)$-conjugacy class of $u_{ij}$ (where it is non-zero) is an invariant of the orbit $O_{V_G}$.

\subsection{Proof of \thmref{cong}}
\begin{proof}
Observe that each orbit $O_{V_G}$ is  determined up to $\Sp(1)$ conjugation of  the semi-normalised Gram matrix represented by $V_{G}$. 

\medskip Let $G = (g_{ij})$ be a semi-normalised Gram matrix represented by $V_{G}$ corresponding to a chosen lift 
$\mathfrak p=(\p_1,\p_2,\ldots,\p_m)$ of $p$. We have the following equations:
$$\A_{23}=\A(p_1, p_2, p_3) = \arccos\frac{\Re(-\langle \p_1, \p_3, \p_2\rangle)}{|\langle \p_1, \p_3, \p_2 \rangle|} = \arccos\frac{\Re(-g_{23})}{| g_{23}|}= \arccos{\Re(-g_{23})},$$

$$\X_{1j'}= g_{23}r_{1j'}{g_{2j'}}^{-1}, \X_{2j}=\overline{g_{23}} g_{2j} r_{1j}^{-1}, \X_{3j}=\overline{g_{23}}^{-1} g_{3j} r_{1j}^{-1} , \X_{kj}=\overline{g_{2k}}^{-1}  g_{kj} r_{1j}^{-1}, 
u_0= \frac{\Im(g_{23})}{|\Im(g_{23})|},$$
 for $ (i+1) \leq j' \leq m,~ 4 \leq j \leq m,~ 4 \leq k \leq i, ~k<j$.\\
 Also for the negative points, we have the following equations:
$$d_{i_1j_1} = g_{i_1j_1}g_{j_1i_1} = |g_{i_1j_1}|^2 ,\A_{i_1j_1} = \arg(g_{i_1j_1}), u_{i_1j_1}=\frac{\Im(g_{i_1j_1})}{|\Im(g_{i_1j_1})|}.$$
\medskip So, given the Gram matrix $G$, we can determine $u_0$, $\A_{23}$, $\A_{i_1j_1}$,  $\X_{ij}$, $d_{i_1j_1}$, and $u_{i_1j_1}$ by the above equations. Using \lemref{4.1o} and the fact that the angular invariants $\A_{ij}$, distance invariants $d_{ij}$  are independent of choices of the lifts, it  determines $\A_{23}$, $\A_{i_1j_1}$, $d_{i_1j_1}$ and the $\Sp(1)$ congruence class of $F=(u_{0},\ldots,u_{t}, \X_1,\ldots, \X_{d})$. 

It is known that each nonzero quaternion $q$ has the unique polar form $q = |q|e^{u\theta}$.Thus, if we know $|q|$, $\theta$ and $u$ then we will get quaternion number $q$ uniquely. By the definition of $d_{i_1j_1}$ we have $|g_{i_1j_1}| = \sqrt{d_{i_1j_1}}$.
Also by the definition of the angular invariant we get $\A_{i_1j_1} = \arg(g_{i_1j_1})$.  Thus we can determine the matrix $A=(g_{i_1j_1})$ for negative points by $d_{i_1j_1}$, $\A_{i_1j_1}$ and $u_{i_1j_1}$.

\medskip Conversely, let $x=(\mu u_{0}\bar{\mu}, \ldots, \mu u_{t}\bar{\mu}, \mu \X_1 \bar{\mu},\ldots, \mu \X_{d}\bar{\mu})$ be an element from the $\Sp(1)$ congruence class of $F$ for some $\mu \in \Sp(1)$ with respect to some lift $p$.
By the above equations we have $\mu g_{23}\bar{\mu} = {\cos \A_{23}+ \mu u_{0}\bar{\mu} \sin \A_{23}}$ with $u_{0}=\frac{\Im(g_{23})}{|\Im(g_{23})|}$. Also, 
$$\mu g_{2j} \bar{\mu} = \mu g_{23}\bar{\mu}~\mu  \X_{2j}\bar{\mu},~\mu g_{3j}\bar{\mu} = \bar{\mu}~\mu \overline{g_{23}} \bar{\mu}~\mu \X_{3j},$$ $$
\mu g_{kj}\bar{\mu} =\mu \overline{ g_{2k}}\bar{\mu}~\mu \X_{kj}\bar{\mu}~ = \mu \overline{\X_{2k}}\bar{\mu}~\mu \overline{g_{23}}\bar{\mu}~~\mu \X_{kj}\bar{\mu},$$
$$\mu {g_{i_1j_1}}\bar{\mu} = |{g_{i_1j_1}}|e^{\mu u_{i_1j_1}\bar{\mu}\A_{i_1j_1}}= \sqrt{d_{i_1j_1}}e^{\mu u_{i_1j_1}\bar{\mu}\A_{i_1j_1}}, $$ where ${g_{i_1j_1}}$ are the entries of sub-matrix $A$ of matrix $G(\textbf{p})$ in
\ref{3.2}.
So we have element $V_G=(\mu g_{23}\bar{\mu},\mu g_{24}\bar{\mu},\ldots,\mu g_{2m}\bar{\mu},\mu g_{34}\bar{\mu},\ldots,\mu g_{3m}\bar{\mu},\ldots,\mu g_{{m-1}m}\bar{\mu})$ with $\lvert \mu g_{23}\bar{\mu}\rvert = 1.$
 Now using  \lemref{ov}, we can determine the $\PSp(n,1)$ congruence class of $p$.
\end{proof}

\bigskip 
\begin{ack}
We thank the referee for many comments and suggestions. Thanks are also due to John Parker for useful discussions on a preliminary draft of this paper during a visit to the ICTS Bangalore for participating in the programs  ICTS/ggd2017/11 and  ICTS/SGGS2017/11. We thank the ICTS for the hospitality and support during the visit. 

\medskip Kalane thanks a UGC research fellowship for supporting him throughout this project.  Gongopadhyay acknowledges partial support from SERB-DST MATRICS project: \\ MTR/2017/000355.  

\end{ack}

%\bibliographystyle{alpha}
%\bibliography{qpairs}

\end{document}